\documentclass[reqno]{amsart}
\usepackage{enumerate,amsmath,amssymb,stmaryrd,mathtools} 
\usepackage{hyperref}
\usepackage{pstricks,pst-node,pst-coil,pst-plot} 

\theoremstyle{plain}
\newtheorem{maintheorem}{Theorem}

\newtheorem{thm}{Theorem}[section]

\newtheorem{corollary}[thm]{Corollary}

\newtheorem{lemma}[thm]{Lemma}

\newtheorem{proposition}[thm]{Proposition}

\theoremstyle{remark}

\theoremstyle{definition}

\newtheorem{definition}[thm]{Definition}

\newcounter{mnotecount}[section]

\newcommand{\definedas}{\mathrel{\raise.095ex\hbox{\rm :}\mkern-5.2mu=}}

\newcommand{\spin}{\text{\rm Spin}}
\newcommand{\adss}{\text{\rm AdSS}}

\def\epsilon{{\varepsilon}}
\def\phi{{\varphi}}

\let\<\langle 
\let\>\rangle

\newcommand{\bR}{\mathbb{R}}

\newcommand{\bH}{\mathbb{H}}

\newcommand{\cA}{\mathcal{A}}

\newcommand{\cM}{\mathcal{M}}

\newcommand{\cN}{\mathcal{N}}

\newcommand{\cF}{\mathcal{F}}

\renewcommand{\hbar}{\overline{h}}

\newcommand{\gtil}{\widetilde{g}}
\newcommand{\Util}{\widetilde{U}}

\newcommand{\Ctil}{\widetilde{C}}

\newcommand{\phitil}{\widetilde{\phi}}

\newcommand{\util}{\widetilde{u}}
\newcommand{\vtil}{\widetilde{v}}

\newcommand{\kulk}{\owedge} 

\DeclareMathOperator{\tr}{tr}
\DeclareMathOperator{\divg}{div}

\DeclareMathOperator{\arcsinh}{arcsinh}

\newcommand{\pdiff} [2]{\frac{\partial #1}{\partial #2}}

\newcommand{\riem}{\mathcal{R}}

\newcommand{\riemuddd}[4]{\riem^{#1}_{\phantom{#1} #2 #3 #4}}

\newcommand{\ric}{\mathrm{Ric}}
\newcommand{\tlric}{\mathring{\ric}}

\newcommand{\ricdd}[2]{\ric_{#1 #2}}

\newcommand{\scal}{\mathrm{Scal}}

\newcommand{\hscal}{\widehat{\scal}}

\DeclareMathOperator{\hess}{Hess}

\newcommand{\hessdd}[2]{\nabla^2_{#1, #2}}

\newcommand{\tlhess}{\mathring{\mathrm{Hess}}}

\newcommand{\tlhessdd}[2]{\mathring{\nabla}_{#1, #2}}

\begin{document} 


\title
{Asymptotically hyperbolic manifolds with small mass}
 
\author{Mattias Dahl}
\address{Institutionen f\"or Matematik \\
  Kungliga Tekniska H\"ogskolan \\
  100 44 Stockholm \\
  Sweden} \email{dahl@math.kth.se}

\author{Romain Gicquaud}
\address{Laboratoire de Math\'ematiques et de Physique Th\'eorique \\
  UFR Sciences et Technologie \\
  Facult\'e Fran\c cois Rabelais \\
  Parc de Grandmont \\
  37200 Tours \\
  France} \email{romain.gicquaud@lmpt.univ-tours.fr}

\author{Anna Sakovich}
\address{Institutionen f\"or Matematik \\
  Kungliga Tekniska H\"ogskolan \\
  100 44 Stockholm \\
  Sweden} \email{sakovich@math.kth.se}

\begin{abstract}
For asymptotically hyperbolic manifolds of dimension $n$ with scalar
curvature at least equal to $-n(n-1)$ the conjectured positive mass 
theorem states that the mass is non-negative, and vanishes only if 
the manifold is isometric to hyperbolic space. In this paper we 
study asymptotically hyperbolic manifolds which are also conformally 
hyperbolic outside a ball of fixed radius, and for which the positive
mass theorem holds. For such manifolds we show that the conformal 
factor tends to one as the mass tends to zero.
\end{abstract}

%
%

\date{\today}


\maketitle

\tableofcontents

\section{Introduction}


The mass of an asymptotically hyperbolic Riemannian manifold is a 
geometric invariant which has been introduced by Wang \cite{WangMass} 
and Chru\'sciel and Herzlich \cite{ChruscielHerzlich} using 
different approaches. The mass is computed in a fixed asymptotically 
hyperbolic end and gives a measure of the leading order deviation of
the geometry from a hyperbolic background metric in the end. For the 
family of anti-de~Sitter-Schwarzschild metrics the mass coincides with 
the mass parameter.

In both papers mentioned above, a positive mass theorem is proved for 
spin manifolds using an adaptation of Witten's spinor argument 
\cite{Witten81}. This theorem states that a complete asymptotically 
hyperbolic spin manifold of dimension $n$ must have non-negative mass
if its scalar curvature is at least equal to $-n(n-1)$ (which is the 
scalar curvature of hyperbolic space of the same dimension). Previous 
work in the physics literature include \cite{AbbottDeser82}, 
\cite{Gibbons.et.al83}, \cite{AshtekarMagnon84}. 

The positive mass theorem also contains a rigidity statement saying 
that the mass vanishes if and only if the manifold is isometric to 
hyperbolic space. In Witten's spinor argument the rigidity follows 
from the fact that vanishing mass forces a certain spinor field to 
satisfy the overdetermined Killing equation, which implies that the 
manifold is hyperbolic. Without the spin assumption the positive mass 
theorem for asymptotically hyperbolic manifolds is still open. 
Partial results have been obtained by Andersson, Cai, and Galloway in 
\cite{AnderssonCaiGalloway} where an adaptation of the minimal surface 
method of \cite{SchoenYau79} is used, see also 
\cite[Section $5\frac{5}{6}$]{Gromov96}. In \cite{AnderssonCaiGalloway} 
the rigidity is proved by first showing 
that the manifold is Einstein. This is done by an argument involving 
a deformation of the metric by the traceless Ricci tensor, if this 
is non-zero one can deform to a metric with strictly negative mass 
which gives a contradiction.

With the rigidity statement of the positive mass theorem in mind, it
is natural to ask what happens if the mass is close to zero and the
scalar curvature is at least equal to $-n(n-1)$. Must the manifold then 
be close to hyperbolic space in some appropriate sense? Such a statement 
can never hold true globally, as the example of the 
anti-de~Sitter-Schwarzschild metric shows. 

The same question has been addressed in relation to the rigidity part 
of the positive mass theorem for asymptotically Euclidean manifolds:
must an asymptotically Euclidean manifold with small mass and 
non-negative scalar curvature be close to Euclidean space in some sense? 
Asymptotically Euclidean spin manifolds with small mass have been 
studied by Bray and Finster, see \cite{BrayFinster02} and
\cite{Finster09}. From estimates on the spinor field in Witten's 
argument they find that the $L^2$-norm of the curvature tensor (over 
the manifold minus an exceptional set) is bounded in terms of the mass. 
Lee \cite{LeeNearEquality} studies asymptotically Euclidean manifolds 
which are conformally flat outside a compact set $K$. For such manifolds 
he proves that the conformal factor can be controlled by the mass, so 
that the conformal factor tends uniformly to one outside any ball
containing $K$ as the mass tends to zero. The argument by Lee does not
require the manifold to be spin, but it needs the assumption that the
positive mass theorem holds for any asymptotically Euclidean metric on
the manifold. 

In the present paper we will adapt the ideas of Lee to the setting 
of asymptotically hyperbolic manifolds. We define a class $\cA(R_0)$
of $n$-dimensional asymptotically hyperbolic manifolds $(M, g)$ which
have scalar curvature greater than or equal to $-n(n-1)$ and have a 
chart at infinity 
$\Phi: M \setminus K \to \bH^n \setminus \overline{B}_{R_0}$, where $K$ 
is a compact subset of $M$ and $R_0$ is a given fixed radius. We require 
that $\Phi_* g$ is conformal to the hyperbolic metric, that is
\[
\Phi_* g = U^\kappa b,
\] 
where $\kappa = \frac{4}{n-2}$, and $\scal^g = -n(n-1)$ on 
$M \setminus K$. Further, we assume that the positive mass theorem 
holds for any asymptotically hyperbolic metric on the manifold $M$. 
We prove that given any $\epsilon > 0$, there is $\delta > 0$ 
such that if a metric belongs to the class $\cA(R_0)$ and has mass 
$m < \delta$ then the conformal factor $U$ satisfies 
$|U-1| < \epsilon$. We refer the reader to Definition \ref{defcA} and 
Theorem \ref{thmNearPMT1} for precise statements of the results.

The most stringent assumption of our theorem is probably that the metric
must be conformal to the hyperbolic metric outside a compact subset.
However, in Appendix \ref{secDensity} we prove that every asymptotically
hyperbolic manifold of scalar curvature greater than or equal to 
$-n(n-1)$ can be approximated by metrics which are conformal to the 
hyperbolic metric outside a ball while changing its mass arbitrarily 
little. See Proposition \ref{propDensity} for the precise statement. 
This result generalizes a proposition of Chru\'sciel and Delay,
\cite[Proposition 6.2]{ChruscielDelay}.

The overall strategy of the proof of Theorem \ref{thmNearPMT1} is as 
follows. We define a 1-parameter family of asymptotically hyperbolic 
metrics involving a geometric property of $(M,g)$, and we compute their 
mass. If the mass of $(M, g)$ is close to zero and if it varies too 
widely with respect to the parameter, this yields a contradiction with 
the positive mass theorem. These ideas are inspired by 
\cite{LeeNearEquality}. However, several complications arise in the 
asymptotically hyperbolic context.

One complication is to reduce the proof of the main theorem to the
case of metrics with constant scalar curvature. This is achieved in
Proposition \ref{propAtoA0} by a conformal transformation of the metric. 
In the asymptotically Euclidean context there is a simple formula for 
the change of mass under a conformal transformation of the metric (see
for example \cite[Lemma 2.1]{LeeNearEquality}), which works nicely 
together with the equation for vanishing scalar curvature. The 
corresponding formula in the asymptotically hyperbolic case is not as 
easily combined with the Yamabe equation for constant scalar curvature.
However, in Proposition \ref{propAtoA0} we give an estimate for the 
difference between the two masses in terms of the respective conformal 
factors.

Once this reduction has been done, we can assume that the metrics we are
considering have constant scalar curvature $\scal^g = -n(n-1)$. A second
complication we encounter is to find an appropriate 1-parameter
family of metrics. We want a deformation that can be localized in the
asymptotic region where the metric is conformal to the hyperbolic metric.
In the view of \cite{LeeNearEquality} and \cite{AnderssonCaiGalloway}, a
natural choice would be $\lambda_s = (\phi_s)^\kappa (g - s \chi \tlric)$,
where $\tlric = \ric + (n-1)g$ is the traceless part of the Ricci tensor,
$\chi$ is a cut-off function whose support is contained in the asymptotic
region, and $\phi_s$ is a conformal factor such that the metrics
$\lambda_s$ have constant scalar curvature $-n(n-1)$. However, with this 
choice the formula for the derivative of the mass turns out to be 
tractable only if $\chi \equiv 1$. Interestingly, this difficulty can be 
overcome by replacing $\tlric$ with a tensor measuring how far the metric 
$g$ is from being static, see Lemma \ref{lmVariationMass}.

We also give a simpler proof for spin manifolds, see Theorem 
\ref{thmNearPMT2}. This argument is based on the fact that the mass 
controls a certain functional which measures how close $(M,g)$ is to 
allow a Killing spinor, and this functional in turn depends continuously
on the conformal factor $U$.

The small mass theorem of Lee \cite{LeeNearEquality} appears as an 
ingredient in the proof of the Penrose inequality by Bray 
\cite{BrayPenrose} and Bray and Lee \cite{BrayLee}. In a forthcoming 
work we plan to address an adaptation of Bray's proof of the Penrose 
inequality to the case of asymptotically hyperbolic manifolds. Note
however that the necessity to replace $\tlric$ by a more complicated
tensor in the definition of the 1-parameter family of metrics sheds 
light on what could be the analog of Bray's conformal flow on 
asymptotically hyperbolic manifolds. Even in the purely Riemannian 
context, the lapse function is likely to play an important role in its 
definition.

This paper is organized as follows. In Section \ref{secPreliminaries} 
we give the definitions of asymptotically hyperbolic manifolds and their 
mass. Section \ref{secGeneralcase} begins with the statement of our main 
result, Theorem \ref{thmNearPMT1}. In the first subsection we prove
some results on the conformal factors at infinity for manifolds in 
$\cA(R_0)$. In the second subsection we then give the proof of our main 
theorem deferring parts of the argument to the following subsections. 
The third subsection contains the argument to show that we can reduce 
to the case $\scal = -n(n-1)$ everywhere by a conformal change while 
controlling the mass. The fourth and final subsection contains the 
proofs of the more technical lemmas. In Section \ref{secSpincase} we give 
the alternative argument for spin manifolds. In Appendix \ref{secAdS} we 
collect details of the anti-de~Sitter-Schwarzschild metric which are used 
in the paper. Finally, in Appendix \ref{secDensity}, we prove Proposition
\ref{propDensity} which shows that metrics which satisfy the assumptions
of Theorems \ref{thmNearPMT1} and \ref{thmNearPMT2} are dense in the set
of metrics which satisfy the standard assumptions of the positive mass
theorem.

\subsection*{Acknowledgments}

We thank Julien Cortier and Marc Herzlich for helpful comments on a
preliminary version of this article. We are also grateful to the 
referees for their careful reading of an earlier version of the 
article. Their insightful comments has led to many improvements of 
the presentation of our results.

\section{Preliminaries}
\label{secPreliminaries}


\subsection{The mass of an asymptotically hyperbolic manifold}

Following the work of Chru{\'s}ciel and Herzlich, 
\cite{ChruscielHerzlich} and \cite{HerzlichMassFormulae}, we define 
the mass of an asymptotically hyperbolic manifold. For conformally 
compact manifolds the definition of the asymptotically hyperbolic mass
coincides with the mass introduced by Wang in \cite{WangMass}. In this 
paper we denote $n$-dimensional hyperbolic space by $\bH^n$ and its 
metric is denoted by $b$. We fix a point in $\bH^n$ as origin. In 
polar coordinates around this point we have 
$b = dr^2 + \sinh^2 r \sigma$ on $(0,\infty) \times S^{n-1}$ where 
$\sigma$ denotes the standard round metric on $S^{n-1}$ and $r$ is the
distance from the origin. The open ball of radius $R$ centered at the
origin is denoted by $B_R$ and its closure is denoted by $\overline{B}_R$.

Let 
$\cN \definedas \{ V \in C^{\infty}(\bH^n) \mid \hess^b V = V b \}$.
This is a vector space with a basis consisting of the functions 
\[
V_{(0)} = \cosh r, \, 
V_{(1)} = x^1 \sinh r, \dots , \,
V_{(n)} = x^n \sinh r,
\]
where the functions $x^1, \dots, x^n$ are the coordinate functions
on $\bR^{n}$ restricted to $S^{n-1}$. The vector space $\cN$ is equipped 
with an inner product $\eta$ of Lorentzian signature characterized by 
the condition that the basis above is orthonormal: 
$\eta(V_{(0)}, V_{(0)}) = 1$, and $\eta(V_{(i)}, V_{(i)}) = -1$ for 
$i=1,\dots,n$. We give $\cN$ a time orientation by specifying the vector 
$V_{(0)}$ to be future directed. The subset $\cN^+$ of positive functions 
then coincides with the interior of the future lightcone. We also  
denote by $\cN^1$ the subset of $\cN^+$ consisting of functions $V$ with 
$\eta(V,V) = 1$. In other words, $\cN^1$ is the unit hyperboloid in the 
future lightcone of $\cN$. For a point $p_0 \in \bH^n$ the function
\[
V \definedas \cosh d_b(p, \cdot)
\] 
is in $\cN^1$, and any function in $\cN^1$ can be given in this form.

A Riemannian manifold $(M,g)$ is called {\em asymptotically hyperbolic} 
if there is a compact subset $K \subset M$ and a diffeomorphism 
$\Phi : M \setminus K \to \bH^n \setminus \overline{B}_R$ for which 
$\Phi_* g$ and $b$ are uniformly equivalent on $\bH^n \setminus B_R$ and
\begin{subequations}
\begin{equation} \label{decay1}
\int_{\bH^n \setminus B_R} 
\left( 
| e |^2 + | \nabla^b e |^2
\right) \cosh r \, d\mu^b < \infty, 
\end{equation}
\begin{equation} \label{decay2}
\int_{\bH^n \setminus B_R} |\scal^g + n(n-1)| \cosh r \, d\mu^b < \infty, 
\end{equation}
\end{subequations}
where $e \definedas \Phi_* g - b$ and $r$ is the (hyperbolic) distance
from an arbitrary given point in $\bH^n$. The diffeomorphism $\Phi$
is also called a chart, or a set of coordinates, at infinity.

The linear functional $H_{\Phi}$ on $\cN$ defined by
\[
H_{\Phi} (V) 
= H_{\Phi}^g (V)
= \lim_{r \to \infty} \int_{S_r} \left(
V (\divg^b e- d \tr^b e) + (\tr^b e) dV - e(\nabla^b V, \cdot)
\right) (\nu_r) \, d \mu^b
\] 
is called the {\em mass functional} of $(M,g)$ with respect 
to $\Phi$. Proposition~2.2 of \cite{ChruscielHerzlich} tells us that 
the limit involved in the definition of $H_{\Phi}$ exists and is finite 
when the decay conditions \eqref{decay1}-\eqref{decay2} are satisfied. 
If $\Phi$ is a chart at infinity as above and $A$ is an isometry of the 
hyperbolic metric $b$ then $A \circ \Phi$ is again a chart at infinity 
and it is not complicated to check that
\[
H_{A \circ \Phi} (V) = H_{\Phi} (V \circ A^{-1}) .
\]
If $\Phi_1$, $\Phi_2$ are charts at infinity as above, then 
\cite[Theorem~2.3]{HerzlichMassFormulae} tells us that there is an 
isometry $A$ of $b$ so that $\Phi_2 = A \circ \Phi_1$ modulo lower order 
terms which do not affect the mass functional.

The mass functional $H_{\Phi}$ is timelike future directed if 
$H_{\Phi}(V) > 0$ for all $V \in \cN^+$. In this case the {\em mass} 
of the asymptotically hyperbolic manifold $(M,g)$ is defined by
\[
m^g \definedas \frac{1}{2(n-1)\omega_{n-1}} \inf_{\cN^1} H^g_{\Phi}(V).
\]
Here $\omega_{n-1}$ denotes the volume of the sphere $(S^{n-1}, \sigma)$.
The factor in front of the infimum is such that the mass of the 
space-like slice 
\[
g_{\adss} = 
\frac{d\rho^2}{1+\rho^2 - \frac{2m}{\rho^{n-2}}} + \rho^2 \sigma
\]
of the anti-de~Sitter-Schwarzschild metric is equal to the parameter 
$m$ in the metric. Note that Chru{\'s}ciel and Herzlich 
\cite[(3.5) and (3.6)]{ChruscielHerzlich} define $m^g$ without this factor.
If $H^g_{\Phi}$ is timelike future directed we may replace the coordinates
at infinity $\Phi$ by $A \circ \Phi$ for a suitably chosen isometry $A$ 
so that $m^g = \frac{1}{2(n-1)\omega_{n-1}} H^g_{\Phi} (V_{(0)})$. Coordinates 
with this property are called {\em balanced}.

The positive mass theorem for asymptotically hyperbolic manifolds, 
\cite[Theorem~4.1]{ChruscielHerzlich} and \cite[Theorem~1.1]{WangMass}, 
states that the mass functional is timelike future directed or zero 
for complete asymptotically hyperbolic spin manifolds with scalar 
curvature $\scal \geq -n(n-1)$. In 
\cite[Theorem~1.3]{AnderssonCaiGalloway} the same result is proved 
with the spin assumption replaced by assumptions on the dimension 
and on the geometry at infinity.

\subsection{Conformally hyperbolic metrics}

We now compute the mass functional of a metric $g$ which is 
asymptotically hyperbolic and conformal to the hyperbolic metric in 
the chart at infinity. That is $\Phi_* g = U^\kappa b$ where $U$ is a 
positive function and we set $\kappa \definedas \frac{4}{n-2}$ as we 
do throughout the paper. In this case $e = f b$ where 
$f \definedas U^\kappa - 1$. The metric $g$ is asymptotically 
hyperbolic if $e$ satisifies \eqref{decay1}-\eqref{decay2}, which turns 
into weighted integral conditions on $U$ and its first two derivatives. 
The mass functional becomes 
\[
H^g_{\Phi} (V)
=
(n-1) \lim_{r \to \infty} 
\int_{S_r} \left( f \partial_r V - V \partial_r f \right) \, d \mu^b .
\] 

If $g$ has constant scalar curvature $-n(n-1)$, so that $U$ is a solution 
to the Yamabe equation, it is known from \cite{AnderssonChruscielFriedrich}
that $U$ has the expansion at infinity
\[
U = 1 + \frac{2^n}{n+1} v e^{-nr} + O( e^{-(n+1)r} ) 
\]
in polar coordinates, where $v$ is a function on $S^{n-1}$. Then 
\[
H^g_{\Phi} \left( \sum_{i=0}^n a_i V_{(i)} \right) = 
\frac{4(n-1)}{n-2}
\int_{S^{n-1}} 
\left( a_0 + \sum_{i=1}^n a_i x^i \right) v 
\, d\mu^{\sigma}, 
\]
and in particular we have 
\begin{equation} \label{confmassbalanced}
m^g 
\leq 
\frac{1}{2(n-1)\omega_{n-1}} H^g_{\Phi} (V_{(0)})
= 
\frac{2}{(n-2)\omega_{n-1}} \int_{S^{n-1}} v \, d\mu^{\sigma} 
\end{equation}
where equality holds if $\Phi$ is a balanced chart at infinity.

\section{Asymptotically hyperbolic manifolds with small mass}
\label{secGeneralcase}


In this section, we prove an analog of the main result of 
\cite{LeeNearEquality}. We first introduce the following class of 
asymptotically hyperbolic manifolds.

\begin{definition} \label{defcA}
For $R_0 > 0$ we let $\cA(R_0)$ be the class of 4-tuples
$(M, g, \Phi, U)$ such that
\begin{itemize}
\item 
$(M, g)$ is a complete Riemannian manifold which is asymptotically 
hyperbolic with respect to $\Phi$, where $\Phi$ is a diffeomorphism 
from the exterior of a compact set $K \subset M$ to 
$\bH^n \setminus \overline{B}_{R_0}$;
\item 
$\scal^g \geq -n(n-1)$, and $\scal^g = -n(n-1)$ on $M \setminus K$;
\item
$U$ is a positive function on $\bH^n \setminus B_{R_0}$ such 
that $U \to 1$ at infinity and $\Phi_* g = U^\kappa b$; 
\item
the coordinates at infinity $\Phi$ are balanced;
\item 
the positive mass theorem holds for any asymptotically hyperbolic 
metric on $M$.
\end{itemize}
\end{definition}

We will prove the following theorem concerning the near-equality case 
for the positive mass theorem.

\begin{maintheorem} \label{thmNearPMT1}
Let $R_1 > R_0$ and $\epsilon > 0$. There is a constant $\delta > 0$
so that  
\[
|U-1| \leq \epsilon e^{-nr}
\]
on $\bH^n \setminus B_{R_1}$ for all 
$(M, g, \Phi, U) \in \cA(R_0)$ with $m^g < \delta$.
\end{maintheorem}

We fix once and for all the value of $R_0$ and abbreviate $\cA =
\cA(R_0)$.

\subsection{A priori estimates}

We first prove estimates on the conformal factor $U$ which are valid for
any element of $\cA$.

\begin{lemma} \label{lmUniformBound}
There are positive constants $A$, $A_k$, $k=0,1,\ldots$, such that for
any $(M, g, \Phi, U)$ belonging to the class $\cA$ we have 
\[\begin{aligned}
\frac{1}{A} \leq U &\leq A, \\
\left| \nabla^{(k)}(U-1) \right| &\leq A_k e^{-n r} 
\quad \text{for } k \geq 0,
\end{aligned}\]
on $\bH^n \setminus B_{R_1}$.
\end{lemma}

Note that these estimates are specific to the case of asymptotically
hyperbolic geometry. In the Euclidean context they cannot be true due to
the fact that the Yamabe equation (which is then the Laplace equation)
is linear.

\begin{proof}
The assumption on the scalar curvature of $\Phi_* g = U^\kappa b$ on 
$\bH^n \setminus B_{R_0}$ implies that $U$ solves the Yamabe equation
\begin{equation} \label{eqYamabeb}
-\frac{4(n-1)}{n-2} \Delta^b U - n(n-1) U = -n (n-1) U^{\kappa+1}
\end{equation}
on $\bH^n \setminus B_{R_0}$. From Propositions \ref{propAdSSchPos1} and
\ref{propAdSSchPos2} we know that there exists a solution $U_+$ 
of Equation \eqref{eqYamabeb} on $\bH^n \setminus \overline{B}_{R_0}$ 
such that $U_+ = 1 + O(e^{-n r})$ at infinity and $U_+ \to \infty$ on 
$\partial B_{R_0}$. Now the same argument as in 
\cite[Proposition 3.6]{GicquaudSakovich} can be used to show that
$U \leq U_+$. Namely, the substitution $U = e^\phi$ brings Equation
\eqref{eqYamabeb} into the form 
\[
-\frac{4(n-1)}{n-2} 
\left(\Delta^b \phi + |d\phi|_b^2\right) - n(n-1)
= -n(n-1)e^{\kappa\phi}.
\]
Subtracting the respective equations for $\phi_+$ and $\phi$ gives
\[
-\frac{4(n-1)}{n-2}\left(\Delta^b (\phi_+-\phi) 
+ \langle d(\phi_+-\phi),d(\phi_++\phi)\rangle_b\right)
+n(n-1)\left(e^{\kappa \phi_+}-e^{\kappa \phi}\right) = 0,
\]
and from the standard maximum principle we conclude that
$\phi_+ \geq \phi$, hence $U_+\geq U$. Similarly, from Proposition
\ref{propAdSSchNeg}, there exists a function $U_-$ such that $U_-$
solves Equation \eqref{eqYamabeb}, $U_- = 1 + O(e^{-n r})$ at infinity,
and $U_- = 0$ on $\partial B_{R_0}$. From the maximum principle we also
conclude that $U_- \leq U$.

We can now finish the proof of the lemma. The existence of the constants 
$A$ and $A_0$ follows from the fact that $U_- \leq U \leq U_+$ on 
$\bH^n \setminus B_{R_1}$. Finally, since $u =U - 1$ satisfies
\[
-\frac{4(n-1)}{n-2}\Delta^b u 
= -n(n-1) \left((1+u)^{\kappa+1}-1\right) + n(n-1)u
\]
we can apply elliptic regularity in balls of fixed radius as above and 
combine with standard bootstrap arguments to get the existence of 
constants $A_k$ for $k \geq 1$.
\end{proof}

From the estimates in Lemma \ref{lmUniformBound} together with 
\eqref{confmassbalanced} we conclude that the mass of the elements 
of $\cA$ is uniformly bounded.

\begin{corollary} 
There exists a constant $C = C(R_0)$ such that for all elements
$(M, g, \Phi, U)$ belonging to the class $\cA(R_0)$, the mass 
satisfies $m^g \leq C$.
\end{corollary}

The exponential decay stated in Theorem \ref{thmNearPMT1} will 
follow from the next proposition. 

\begin{proposition} \label{propUControlledDecay}
Let $R_1 > R_0$ be a fixed radius. There exists a constant $C>0$
such that for any $(M, g, \Phi, U)$ in the class $\cA$ we have 
\begin{equation} \label{eqUControlledDecay}
|U-1| \leq C \left( \sup_{\bH^n \setminus B_{R_1}} |U-1| \right) e^{-nr}
\end{equation}
on $\bH^n \setminus B_{R_1}$.
\end{proposition}

\begin{proof}
In Appendix \ref{secAdS} we have described the solutions $f_m$ of 
\eqref{eqYamabeb} corresponding to anti-de~Sitter-Schwarzschild metrics
of mass $m$. For appropriate choice of $m_- < 0 < m_+$ we have that 
$f_{m_+}$ and $f_{m_-}$ solve \eqref{eqYamabeb} on $\bH^n \setminus B_{R_0}$
with $f_{m_+} \to \infty$ on $\partial B_{R_0}$ and $f_{m_-} = 0$ on
$\partial B_{R_0}$. From the proof of Lemma \ref{lmUniformBound} we know 
that $f_{m_-} \leq U \leq f_{m_+}$ on $\bH^n \setminus B_{R_0}$. 

Let $0 \leq m \leq m_+$. Then $f_m$ such that $1 \leq f_m \leq f_{m_+}$ 
is defined for $r \geq R_1$, see Appendix \ref{secAdS} for details. From 
the proof of Proposition \ref{propAdSSchPos1} we know that 
$0 \leq f_m - 1 \leq Cme^{-nr}$ for
$r \geq r_1(m) \definedas \max\{R_1, r((2 m)^{1/n})\}$. It is not
complicated to extend this estimate to the whole interval 
$r \geq R_1$. Indeed, let $\mu>0$ be such that $R_1 = r((2\mu)^{1/n})$. 
If $0 \leq m \leq \mu$ then we
have $r_1(\mu)=R_1$, hence the estimate already holds for $r \geq R_1$.
Therefore it suffices to consider the case $\mu < m \leq m_+$ which
corresponds to the situation $r_1(m)> R_1$. Since $f_m$ is decreasing we
have $f_m-1\leq f_m(R_1)-1\leq f_{m_+}(R_1)-1$ on $R_1\leq r \leq r_1(m)$,
whereas $me^{-nr}\geq \mu e^{-n r_1(m)}\geq \mu e^{-nr_1(m_+)}$ on this
interval. It is now clear that up to increasing $C$ if necessary, we
can assume that the inequality $0 \leq f_m-1 \leq Cme^{-nr}$ holds for
$r\geq R_1$. In the rest of the proof, the constant $C>0$ might vary 
from line to line but remains independent of $m$.

Using Proposition \ref{propAdSSchNeg} we can similarly prove that the
inequality $Cme^{-nr} \leq f_m-1 \leq 0$ holds for $r\geq R_1$ in the
case when $m_-\leq m \leq 0$. This yields 
\[
|f_m-1| \leq C|m|e^{-nr} 
\]
for $m_- \leq m \leq m_+$ and $r\geq R_1$. Let us now choose 
$\underline{m},\overline{m} \in (m_-, m_+)$ so that 
$f_{\underline{m}}(R_1)=\inf_{\partial B_{R_1}}U$ and
$f_{\overline{m}}(R_1)=\sup _{\partial B_{R_1}}U$. Again, the use of the
maximum principle as in the proof of Proposition \ref{propAdSSchPos1}
yields $f_{\underline{m}}\leq U \leq f_{\overline{m}}$ on
$\bH^n\setminus B_{R_1}$. Consequently, we have the estimate 
\[
|U-1| \leq 
C \max \left\{ |\underline{m}|,|\overline{m}| \right\} e^{-nr} 
\]
on $\bH^n \setminus B_{R_1}$.

With all these preliminaries at hand, \eqref{eqUControlledDecay} is a
simple consequence of the fact that there exists a constant $C>0$ such
that
\begin{equation}\label{eqMassControl}
|m|\leq C|f_m(R_1)-1| \hspace{0.4cm} \text{ for }
\hspace{0.4cm} m_-\leq m \leq m_+.
\end{equation}
Indeed, if we assume that this estimate holds, then 
\[ \begin{split}
|U-1|
&\leq 
C \max \left\{
\left| f_{\underline{m}}(R_1)-1 \right|, 
\left| f_{\overline{m}}(R_1)-1 \right| 
\right\}e^{-nr}\\
&=
C \max \left\{
\left| \inf_{\partial B_{R_1}} U-1 \right|, 
\left| \sup_{\partial B_{R_1}} U-1 \right|
\right\}e^{-nr}\\
&=
C \max 
\left\{
\left| \inf_{\partial B_{R_1}} (U-1) \right|, 
\left| \sup_{\partial B_{R_1}} (U-1) \right| 
\right\}e^{-nr}\\
&\leq 
C \left( \sup_{\partial B_{R_1}} |U-1| \right) e^{-nr}\\
&\leq 
C \left( \sup_{\bH^n\setminus B_{R_1}} |U-1| \right) e^{-nr}.
\end{split} \]
Consequently, in order to complete the proof, we only need to prove
\eqref{eqMassControl}. In fact, \eqref{eqMassControl} will follow  
from the monotonicity of $f_m$ if we show that
\begin{equation}\label{eqControlM}
|m| \leq C|f_m(R_2)-1| \hspace{0.4cm} \text{ for }
\hspace{0.4cm} m_-\leq m \leq m_+,
\end{equation}
for some $R_2>R_1$. We fix $R_2 > \max\{r_0(m_+),R_1\}$ and set
$x \definedas f_m(R_2)$. It is clear that
$f_{m_-}(R_2)\leq x \leq f_{m_+}(R_2)$ for $m_-\leq m \leq m_+$, and that
$r^{-1}(R_2)=x^{\frac{2}{n-2}}\sinh R_2>a(m)$. Then \eqref{eqRelationRhoR}
yields
\[
\int_{x^{\frac{2}{n-2}}\sinh R_2}^\infty
\frac{d\rho}{\rho\sqrt{1+\rho^2 - \frac{2m}{\rho^{n-2}}}}
= \int_{R_2}^\infty \frac{dr}{\sinh r}.
\]
We define
\[
F(x,m) \definedas 
\int_{x^{\frac{2}{n-2}}\sinh R_2}^\infty
\frac{d\rho}{\rho\sqrt{1+\rho^2 - \frac{2m}{\rho^{n-2}}}},
\] 
where $f_{m_-}(R_2)\leq x \leq f_{m_+}(R_2)$, $m_-\leq m \leq m_+$.
It is straightforward to check that 
\[\begin{split}
\frac{\partial F}{\partial m}
&=
\int_{x^{\frac{2}{n-2}}\sinh R_2}^\infty 
\frac{d\rho}{2\rho^{n-1}\left(1+\rho^2-\frac{2m}{\rho^{n-2}}\right)^{3/2}}\\
&\geq 
\int_{f_{m_+(R_2)}^{\frac{2}{n-2}}\sinh R_2}^\infty 
\frac{d\rho}{2\rho^{n-1}\left(1+\rho^2-\frac{2m_-}{\rho^{n-2}}\right)^{3/2}}
\end{split}\] 
is positive and uniformly bounded away from zero, and that
\[
\frac{\partial F}{\partial x} 
= 
-\frac{2}{(n-2)x\sqrt{1+x^{\frac{4}{n-2}}(\sinh R_2)^2
-\frac{2m}{x^2 (\sinh R_2)^{n-2}}}}
\]
is uniformly bounded. We conclude that there exists $C>0$ such that 
$|m'(x)|<C$ for $x\in (f_{m_-}(R_2), f_{m_+}(R_2))$. Finally, applying the 
mean value theorem we arrive at \eqref{eqControlM} and thus 
\eqref{eqMassControl} follows.
\end{proof}

\begin{corollary}  
There exists a radius $R_2 > R_1$ such that for $(M, g, \Phi, U) \in \cA$
the function $|U-1|$ reaches its maximum over 
$\bH^n \setminus B_{R_1}$ in the annulus 
$A_{R_1, R_2} = \overline{B}_{R_2} \setminus B_{R_1}$.
\end{corollary}

\begin{proof}
Choose $R_2$ such that $C e^{-n R_2} \leq 1$. Then for any point such 
that $r > R_2$ we have 
\[
|U-1| 
\leq C \left( \sup_{\bH^n \setminus B_{R_1}} |U-1| \right) e^{-nr} 
< \sup_{\bH^n \setminus B_{R_1}} |U-1|.
\]
\end{proof}

\subsection{Strategy of the proof of Theorem \ref{thmNearPMT1}}
\label{secStrategy}

In this subsection we discuss the main strategy of the proof of Theorem
\ref{thmNearPMT1}, deferring the proof of technical details to the next 
subsections.

The first step is to reduce the proof of Theorem \ref{thmNearPMT1} to the 
particular case of metrics with constant scalar curvature 
$\scal^g = -n(n-1)$. For this we show that the conformal factor 
transforming the metric $g$ to a metric with constant scalar curvature
can be uniformly controlled on $\bH^n \setminus B_{R_1}$ by the difference
between the masses (more exactly of the time components
$H_{\Phi}(V_{(0)})$ of the mass functional) of the two metrics. This is
the content of the following proposition.

\begin{proposition} \label{propAtoA0}
Given $(M, g, \Phi, U) \in \cA$, there exists a unique positive 
function $w$ on $M$ such that $\gtil \definedas w^{\kappa} g$ is 
asymptotically hyperbolic with constant scalar curvature 
$\scal^{\gtil} = -n(n-1)$. 
The metric $\gtil$ has mass $m^{\gtil} \leq m^g$. Further, for 
$p > n/2$ there is a constant $C > 0$ independent of 
$(M, g, \Phi, U)$ such that 
\[
\sup_{\bH^n \setminus B_{R_1}} \left| U - \Util \right| 
\leq C \left( m^g - m^{\gtil} \right)^{1/p},
\]
where $\Util \definedas U w$.
\end{proposition}

This reduction turns out to be convenient for obtaining estimates in the
second part of the proof. We introduce the restricted class $\cA_0(R_0)$
of 4-tuples $(M, g, \Phi, U) \in \cA$ such that $\scal^g = -n(n-1)$ on all
of $M$. To prove Theorem \ref{thmNearPMT1} we need to show the result
for elements of $\cA_0 = \cA_0(R_0)$.

The basic idea is to apply the positive mass theorem to a certain 
1-parameter family of metrics. To define it, we first modify the metric 
$g$ in an annulus (see Equation \eqref{eqDefGs}) and conformally transform 
it to fulfill the assumption $\scal \geq -n(n-1)$ of the positive mass 
theorem.

In the first lemma we prove the existence of a function $V$ which solves
$\Delta^g V = n V$ and which is asymptotic to $V_{(0)}$. For functions $V_1$ and 
$V_2$ on $M$ we write $V_1 \sim V_2$ if $V_1/V_2$ tends to $1$ at infinity.

Let $R'_0, R''_0, R'_1$ and $R''_1$ be constants such that
\[
R_0 < R'_0 < R''_0 < R_1 < R'_1 < R''_1.
\]
We remind the reader that $r$ denotes the distance function from the 
chosen origin in $\bH^n$.

\begin{lemma} \label{lmBoundV}
Let $(M, g, \Phi, U) \in \cA$. There exists a unique solution 
$V^g$ to the equation
\begin{equation} \label{eqEigenVect}
\Delta^g V = n V
\end{equation}
such that $V^g \sim V_{(0)}$. Further, there exist universal functions 
\[
V_{\pm} : \bH^n \setminus B_{R'_0} \to \bR
\]
such that for some constants $C_0, C_1 > 0$ we have
\[
\left| V_{\pm} - V_{(0)} \right| \leq C_0 e^{-(n-1) r},
\]
\[V_- \leq V^g \leq V_+,\]
and
\[\left| dV^g - dV_{(0)} \right|_g \leq C_1 e^{-(n-1) r}\]
on $\bH^n \setminus B_{R'_0}$. Also, there are constants
$B_2, B_3, \ldots$ depending only on $R'_0, R''_0$ and $R''_1$ such that
for any integer $k \geq 2$ we have
\[
\left| \nabla^{(k)} V^g \right| \leq B_k \text{ on } A_{R''_0, R''_1}.
\]
\end{lemma}

Define
\[
T \definedas \tlric^g - \frac{\tlhess^g V^g}{V^g},
\]
where $\tlric^g = \ric^g + (n-1)g$ denotes the traceless part of the 
Ricci tensor and $\tlhess^g V = \hess^g V - V g$ denotes the traceless 
part of the Hessian of $V$. From the computations in the proof of 
Lemma \ref{lmVariationMass} it follows that the tensor $V^g T$ 
is actually the gradient of the mass at $(M, g)$ in the space $\cA_0(R_1'')$.

We choose a smooth function $\chi$ such that
\[
\chi = 
\begin{cases}
0 & \text{on $B_{R''_0}$}, \\
1 & \text{on $A_{R_1, R'_1}$}, \\
0 & \text{on $\bH^n \setminus B_{R''_1}$},
\end{cases}
\]
and define the metric
\begin{equation}\label{eqDefGs}
g_s \definedas g + s \chi T
\end{equation}
for small values of the parameter $s$.

Next we recall the definition of the weighted local Sobolev spaces,
see \cite{GicquaudSakovich} for more details on these spaces. Let 
$p \in (1, \infty)$,  a non negative integer $k$, and $\delta \in \bR$ 
be given. We define the function space $X^{k, p}_\delta(M, \bR)$ as the 
set of functions $u \in W^{k, p}_{loc}(M, \bR)$ such that the norm
\begin{equation} \label{eqDefXkp}
\left\|u \right\|_{X^{k, p}_\delta(M, \bR)} 
= \sup_{x \in M} e^{\delta r(x)} \left\|u \right\|_{W^{k, p}(B_1(x), \bR)}
\end{equation}
is finite. This space is a Banach space.

We will conformally transform the metrics $g_s$ to have constant scalar 
curvature $\scal = -n(n-1)$. The details of this are taken care of in
the following lemma.

\begin{lemma} \label{lmConformalFactor}
There exists $s_0 > 0$ such that for all $s \in [-s_0, s_0]$ and any 
$(M, g, \Phi, U) \in \cA_0$ it holds that
\[
\frac{1}{2} g \leq g_s \leq 2 g
\]
and
\[
\left| \scal^{g_s} + n(n-1) \right| \leq n-1.
\]
Further, for any $s \in [-s_0, s_0]$ there exists a unique positive function 
$\phi_s$ on $M$ which is bounded from above and away from zero such that 
the metric
\begin{equation*}  
\lambda_s \definedas \phi_s^{\kappa} g_s
\end{equation*}
has constant scalar curvature $-n(n-1)$. The function $\phi_s$ satisfies
\[
\left(\frac{n-1}{n}\right)^{1/\kappa} \leq \phi_s \leq
\left(\frac{n+1}{n}\right)^{1/\kappa}.
\]
In addition, there are constants $C_0, C_1, \ldots$ such that
\begin{equation}\label{eqEstimatesPhi}
\left|\nabla^{(k)}(\phi_s - 1) \right| \leq C_k e^{-n r}
\end{equation}
holds on $\bH^n \setminus B_{R_1}$ for all $k \geq 0$. Finally, the 
map $s \mapsto \phi_s - 1$ from the interval $[-s_0, s_0]$ to
$X^{2, p}_\delta(M, g)$ is $C^2$ for any $p \in (n, \infty)$ and
$\delta \in \left(\frac{n}{2}, n\right)$.
\end{lemma}

For $V = V_{(0)} = \cosh r$ we set 
$H(s) \definedas H_{\Phi}^{\lambda_s} (V)$. This is the time component of
the mass functional, which gives an upper bound on the mass, namely
$m^{\lambda_s} \leq \frac{1}{2(n-1)\omega_{n-1}} H(s)$. Since the coordinates
at infinity are balanced for $g$, we have 
$m^g = m^{\lambda_0} = \frac{1}{2(n-1)\omega_{n-1}} H(0)$.
In what follows we 
will denote derivatives with respect to the parameter $s$ by a dot.

\begin{lemma}\label{lmMassLambda}
The map $s \mapsto H(s)$ is a $C^2$ function. Further, there is a 
constant $A$ independent of $(M, g, \Phi, U) \in \cA_0$ such that 
\[
|\ddot{H}(s)| \leq A.
\]
\end{lemma}

In the next proposition we find that $\dot{H}(0)$ is related to the 
$L^2$-norm of $\tlric^g - \frac{1}{V^g} \tlhess V^g$ on an annulus,
which can be interpreted as a measure of ``non-staticity'' of the
metric $g$ on the annulus.

\begin{lemma}\label{lmVariationMass} 
Suppose $(M, g, \Phi, U) \in \cA_0$ and $H(s)$ is defined as above, 
then 
\[
\dot{H}(0) = 
\int_M 
\chi V^g \left| \tlric^g - \frac{\tlhess V^g}{V^g} \right|_g^2 
\, d\mu^g.
\]
\end{lemma}

We are now ready to prove Theorem \ref{thmNearPMT1}.

\begin{proof}[Proof of Theorem \ref{thmNearPMT1}]
We first assume that the metric $g$ has constant scalar curvature.
Applying Taylor's formula to $H(s)$ on the interval $(-s_0, s_0)$ 
we find
\[\begin{split}
H(s) 
&= 
H(0) + s \dot{H}(0) + \int_0^s (s-t) \ddot{H}(t) dt\\
&\leq 
H(0) + s \dot{H}(0) + A \int_0^s (s-t) dt\\
&\leq 
H(0) + s \dot{H}(0) + \frac{A}{2} s^2.
\end{split}\]
From the assumption that the positive mass theorem holds for any 
asymptotically hyperbolic metric on $M$ we have 
$H(s) \geq 2(n-1)\omega_{n-1} m^{\lambda_s} \geq 0$ for $s \in (-s_0, s_0)$.
As a consequence,
\[
0 \leq H(0) + s \dot{H}(0) + \frac{A}{2} s^2.
\]
Assuming that $H(0) \leq \frac{2 s_0^2}{A}$, we write the 
previous inequality with $s = -\sqrt{\frac{2 H(0)}{A}}$ and 
get
\begin{equation} \label{eqControlmdot}
\dot{H}(0) 
\leq \sqrt{2 A H(0)}
= \sqrt{4 A (n-1) \omega_{n-1} m^g}. 
\end{equation}

Let $\epsilon$ be an arbitrary positive number. We claim that there 
exists $\delta > 0$ such that any $(M, g, \Phi, U)$ belonging to 
$\cA_0$ and having mass $m^g \leq \delta$ satisfies
\[
\sup_{\bH^n \setminus B_{R_1}} \left|U-1\right| \leq \epsilon.
\]
To prove this we argue by contradiction and assume that there is a 
sequence $(M_k, g_k, \Phi_k, U_k)$ of elements of $\cA_0$ such that 
the mass $m_k \definedas m^{g_k}$ tends to zero while
$\left| U_k - 1 \right| \geq \epsilon$. Using Lemmas 
\ref{lmUniformBound}, \ref{lmBoundV} and 
\cite[Proposition 2.3]{GicquaudSakovich} 
(Rellich theorems for weighted local Sobolev spaces), we construct 
functions $U_\infty$ and $V_\infty$ on $\bH^n \setminus B_{R'_0}$ as 
limits of some subsequence of $U_k$ and $V^{g_k}$. Choose 
$p \in (n, \infty)$ and 
$\delta \in \left(\frac{n}{2}, n\right)$.
\begin{itemize}
\item 
From Lemma \ref{lmUniformBound}, the sequence $U_k - 1$ is
bounded in $X^{3, p}_n(\bH^n \setminus B_{R'_0})$. Hence there exists a
subsequence converging to a limit $U_\infty - 1$ in $X^{2, p}_\delta$.
\item 
To construct $V_\infty$, it suffices to remark that the sequence $V_k$ 
is uniformly bounded in $W^{3, p}(K)$ for any compact subset
$K \subset \bH^n \setminus \mathring{B}_{R'_0}$ by standard elliptic
regularity. Hence, by a diagonal process, we can construct a subsequence
of functions $V_k$ converging in the $W^{2, p}$-norm on any compact subset.
\end{itemize}
The function $U_\infty$ solves \eqref{eqYamabeb} and $V_\infty$ solves 
$\Delta^{g_\infty} V_\infty = n V_\infty$, where 
$g_\infty \definedas U_\infty^\kappa b$. They satisfy the asymptotics of 
Lemmas \ref{lmUniformBound} and \ref{lmBoundV}. Further,
\begin{equation} \label{Uinftynotzero}
\sup_{\bH^n \setminus B_{R_1}} \left| U_{\infty} - 1 \right| \geq \epsilon.
\end{equation} 
The metric $g_\infty$ has mass zero since the mass depends continuously
on $U-1 \in X^{2, p}_\delta$ (see the proof of Lemma \ref{lmMassLambda}).

Lemma \ref{lmVariationMass} together with the estimate 
\eqref{eqControlmdot} applied to $(M_k, g_k, \Phi_k, U_k)$ gives the
inequality
\[
\int_M \chi V^{g_k} 
\left| \tlric^{g_k} - \frac{\tlhess V^{g_k}}{V^{g_k}} \right|_{g_k}^2 
\, d\mu^{g_k}
\leq \sqrt{4 A (n-1) \omega_{n-1} m_k}
\]
for any $k$. In particular, we obtain
\[
\int_{\bH^n \setminus B_{R'_0}} \chi V^{g_\infty} 
\left| 
\tlric^{g_\infty} - \frac{\tlhess V^{g_\infty}}{V^{g_\infty}}
\right|_{g_\infty}^2 
\, d\mu^{g_\infty} = 0
\]
when we let $k$ tend to infinity. Therefore 
\[
\tlric^{g_\infty} = \frac{\tlhess V^{g_\infty}}{V^{g_\infty}}
\] 
on $A_{R''_0, R'_1}$. By analyticity this equality holds on all of 
$\bH^n \setminus B_{R'_0}$. From Proposition \ref{propLCFStatic} and the 
fact that the metric $g_\infty$ has zero mass, we conclude that $g_\infty$ 
is hyperbolic. This forces $U_\infty = 1$ which contradicts 
\eqref{Uinftynotzero}. We have thus proved the claim made above.

At this point, we would like to emphasize that the metric $g_\infty$ is
defined only on $\bH^n \setminus B_{R'_0}$ so it is not complete.
In particular, the standard positive mass theorem does not apply. This
is why Proposition \ref{propLCFStatic} is needed.

The proof of Theorem \ref{thmNearPMT1} in the general case
$\scal^g \geq -n(n-1)$ is then concluded by Proposition \ref{propAtoA0}
followed by Proposition \ref{propUControlledDecay}.
\end{proof}

\subsection{Proof of Proposition \ref{propAtoA0}}

In this section we prove Proposition \ref{propAtoA0}: the conformal factor 
transforming a metric $g$ to a metric $\gtil$ with constant scalar 
curvature is controlled by the difference $m^g - m^{\gtil}$ of their masses. 
This was used to reduce the proof of Theorem \ref{thmNearPMT1} to 
elements of the class $\cA_0$. Such a reduction can also be found in 
\cite[Proposition 3.13]{AnderssonCaiGalloway}.

As it will become apparent, the proof of this proposition yields a simpler
argument for Theorem \ref{thmNearPMT1} in the case $U \geq 1$. However,
since it is based on estimates for solutions to the Yamabe equation on 
$\bH^n \setminus B_{R_1}$, the argument cannot be generalized to 
arbitrary $U$. Indeed, one can find solutions to the Yamabe 
equation \eqref{eqYamabeb} on $\bH^n \setminus B_{R_0}$ that
oscillate around $1$ to produce metrics with zero mass. This shows 
that the strategy of the proof of Proposition \ref{propAtoA0}
is too weak to produce a full proof of Theorem \ref{thmNearPMT1}.

We first make a certain observation about $(M, g, \Phi, U) \in \cA$.
If we set $U = 1 + u$, Equation \eqref{eqYamabeb} can be written in the
form
\begin{equation} \label{eqYamabebu}
\partial_r^2 u + (n-1) \coth r \, \partial_ r u - nu
= 
f(u) - \sinh^{-2} r \Delta^{\sigma} u
\end{equation}
where 
\[
f(u) \definedas
\frac{n(n-2)}{4} 
\left( (1+u)^{\frac{n+2}{n-2}} - 1 - \frac{n+2}{n-2} u \right).
\]
We remark that the ordinary differential equation
\[
u''(r) + (n-1) \coth r\, u'(r) - nu(r) = 0
\]
has the solutions
\begin{align*}
u_0 (r) &= 
\cosh r \int _r ^{\infty} \frac{1}{\cosh^2 \tau \sinh^{n-1} \tau} \, d\tau 
= \frac{2^n}{n+1} e^{-nr} + O(e^{-(n+2)r}), \\
u_1(r) &= \cosh r.
\end{align*}

\begin{lemma} \label{lmIntv}
Suppose $U = 1 + u$ is such that $u$ satisfies \eqref{eqYamabebu} on 
$\bH^n \setminus B_{R_0}$ and the metric $U^{\kappa} b$ is asymptotically 
hyperbolic with respect to the identity chart at infinity.
Then $v \definedas u / u_0$ satisfies 
\[\begin{split}
&\int_{S^{n-1}} v(s) \, d\mu^{\sigma} 
\geq \frac{(n-2)\omega_{n-1}}{2} m^{U^{\kappa} b} \\
&\qquad +
\int_s^{\infty}
\left( 1 - \frac{\cosh s}{\cosh r} \frac{u_0(r)}{u_0(s)} \right)
\cosh r \sinh^{n-1} r 
\left( \int_{S^{n-1}} f(u(r,\theta)) \, d\mu^{\sigma} \right)
\, dr
\end{split}\]
where $ m^{U^{\kappa} b}$ is the mass of the metric $U^{\kappa} b$. 
Equality holds if the identity chart at infinity is balanced for 
$U^{\kappa} b$.
\end{lemma}

\begin{proof}
Substituting $u=u_0 v$ into \eqref{eqYamabebu} we get
\[
u_0 \partial_r^2 v + ( 2 u_0' + (n-1) \coth r\, u_0) \partial_r v 
= 
f(u) - \sinh^{-2} r \Delta^{\sigma} u . 
\]
If we multiply this equation by $u_0 \sinh^{n-1} r$ we obtain
\[
\partial_r \left( u_0^2 \sinh^{n-1} r \, \partial_r v \right)
= 
u_0 \sinh^{n-1} r \left(
f(u) - \sinh^{-2} r \Delta^{\sigma} u
\right). 
\]
Integration from $t$ to $\infty$ gives
\[\begin{split}
&\left( u_0^2 \sinh^{n-1} r\, \partial_r v \right)_{| r = \infty}
-
\left( u_0^2 \sinh^{n-1} r \, \partial_r v \right)_{| r = t} \\ 
&\qquad = 
\int_t^{\infty}
u_0(r) \sinh^{n-1} r \left(
f(u(r,\theta)) - \sinh^{-2} r \Delta^{\sigma} u(r,\theta)
\right) \, dr.
\end{split}\]
We observe that $\partial_r v = O(1)$ by Lemma \ref{lmUniformBound}.
Hence $u_0^2 \sinh^{n-1} r \, \partial_r v = O(e^{-(n+1)r})$, so the first 
term in the left-hand side vanishes. Consequently we have 
\[
- \partial_r v(t,\theta) 
= 
\frac{1}{u_0^2(t) \sinh^{n-1} t}
\int_t^{\infty}
u_0(r) \sinh^{n-1} r \left(
f(u(r,\theta)) - \sinh^{-2} r \Delta^{\sigma} u(r,\theta)
\right) \, dr.
\]
Integrating from $s$ to $\infty$ and changing order of integration we 
obtain
\[\begin{split}
&v(s,\theta) - \lim_{r \to \infty} v(r,\theta) \\
&\quad = 
\int_s^{\infty}
\frac{1}{u_0^2(t) \sinh^{n-1} t}
\int_t^{\infty}
u_0(r) \sinh^{n-1} r \left(
f(u(r,\theta)) - \sinh^{-2} r \Delta^{\sigma} u(r,\theta)
\right) \, dr dt \\
&\quad = 
\int_s^{\infty}
\left( \int_s^r \frac{1}{u_0^2(t) \sinh^{n-1} t} \, dt \right)
u_0(r) \sinh^{n-1} r 
\left( f(u(r,\theta)) - \sinh^{-2} r \Delta^{\sigma} u(r,\theta) \right)
\, dr.
\end{split}\]
Here the integral over $t$ is 
\[\begin{split}
&\int_s^r \frac{1}{u_0^2 (t) \sinh^{n-1} t} \, dt \\
&\qquad =
\int_s^r \frac{1}{\cosh^2 t \sinh^{n-1} t
\left(
\int_t^{\infty} \frac{1}{\cosh^2 \tau \sinh^{n-1} \tau} \, d\tau
\right)^2
} \, dt \\
&\qquad =
\int_s^r 
\left(
\frac{1}
{\int_t^{\infty} \frac{1}{\cosh^2 \tau \sinh^{n-1} \tau} \, d\tau}
\right)'
\, dt \\
&\qquad =
\frac{1}
{\int_r^{\infty} \frac{1}{\cosh^2 \tau \sinh^{n-1} \tau} \, d\tau}
-
\frac{1}
{\int_s^{\infty} \frac{1}{\cosh^2 \tau \sinh^{n-1} \tau} \, d\tau} \\
&\quad =
\frac{\cosh r}{u_0(r)} - \frac{\cosh s}{u_0(s)},
\end{split}\]
thus
\[\begin{split}
&v(s,\theta) - \lim_{r \to \infty} v(r,\theta) \\
&\quad = 
\int_s^{\infty}
\left( \frac{\cosh r}{u_0(r)} - \frac{\cosh s}{u_0(s)} \right)
u_0(r) \sinh^{n-1} r 
\left( f(u(r,\theta)) - \sinh^{-2} r \Delta^{\sigma} u(r,\theta) \right)
\, dr .
\end{split}\]
From \eqref{confmassbalanced} we have 
\[
\frac{(n-2)\omega_{n-1}}{2} m^{U^{\kappa} b}
\leq \lim_{r \to \infty} \int_{S^{n-1}} v(r,\theta) \, d\mu^{\sigma},
\]
so when we integrate over $S^{n-1}$ we arrive at
\[\begin{split}
&\int_{S^{n-1}} v(s,\theta) \, d\mu^{\sigma} 
- \frac{(n-2)\omega_{n-1}}{2} m^{U^{\kappa} b} \\
&\qquad \geq
\int_s^{\infty}
\left( \frac{\cosh r}{u_0(r)} - \frac{\cosh s}{u_0(s)} \right)
u_0(r) \sinh^{n-1} r 
\left( \int_{S^{n-1}} f(u(r,\theta)) \, d\mu^{\sigma} \right)
\, dr \\
&\qquad \geq
\int_s^{\infty}
\left( 1 - \frac{\cosh s}{\cosh r} \frac{u_0(r)}{u_0(s)} \right)
\cosh r \sinh^{n-1} r 
\left( \int_{S^{n-1}} f(u(r,\theta)) \, d\mu^{\sigma} \right)
\, dr,
\end{split}\]
with equality if the coordinates at infinity are balanced. 
\end{proof}

\begin{proof}[Proof of Proposition \ref{propAtoA0}]
The existence of the function $w$ is guaranteed by 
\cite[Theorem 1.2]{AnderssonChruscielFriedrich} which says that any 
asymptotically hyperbolic manifold is conformally related to
one with scalar curvature $-n(n-1)$. The function $w$ is a 
solution of the Yamabe equation
\begin{equation}\label{eqYamabeg}
-\frac{4(n-1)}{n-2} \Delta^g w+ \scal^g w = -n (n-1) w^{\kappa+1}.
\end{equation}
Since $\scal^g \geq -n(n-1)$, the constant function $1$ is a 
supersolution of \eqref{eqYamabeg}. Applying the maximum principle as 
in the proof of Lemma \ref{lmUniformBound} we conclude that $w \leq 1$. 
Consequently, since both $U$ and $\Util$ satisfy the Yamabe equation 
\eqref{eqYamabeb}, it follows from the proof of Lemma 
\ref{lmUniformBound} that $U_- \leq \Util \leq U \leq U_+$ on 
$\bH^n\setminus B_{R_0}$. We set $\util = \Util - 1$, 
$\vtil = u_0^{-1} \util$, and we note that $\util \leq u$ and 
$\vtil \leq v$. Since $\Phi$ are balanced coordinates at infinity 
for $g$ (but not necessarily for $\gtil$) we see from 
\eqref{confmassbalanced} that 
\[
m^g - m^{\gtil} 
\geq 
\lim_{r \to \infty} \frac{2}{(n-2)\omega_{n-1}}
\int_{S^{n-1}} (v(r,\theta) - \vtil(r,\theta)) \, d\mu^{\sigma}
\geq 0.
\]
Again, since $\Phi$ are balanced coordinates at infinity for $g$ we 
conclude from Lemma \ref{lmIntv} that
\[\begin{split}
&\int_{S^{n-1}} (v(s,\theta)-\vtil (s,\theta)) \, d\mu^{\sigma} 
\leq
\frac{(n-2)\omega_{n-1}}{2}
\left( m^g - m^{\gtil} \right) \\
&\qquad + 
\int_s^{\infty}
\left( 1 - \frac{\cosh s}{\cosh r} \frac{u_0(r)}{u_0(s)} \right)
\cosh r \sinh^{n-1} r 
\left( \int_{S^{n-1}} \left(f(u(r,\theta))-f(\util(r,\theta)) \right) 
\, d\mu^{\sigma} \right) \, dr .
\end{split}\]
Observe that
\[
0 \leq \frac{\cosh s}{\cosh r} \frac{u_0(r)}{u_0(s)} \leq 1. 
\]
Moreover, recall that $u_+ = U_+ - 1 > 0$. Therefore we can use mean
value theorem to show that
\[\begin{split}
f(u) - f(\util) 
&= 
f'(tu+(1-t)\util))(u-\util) \\
&\leq 
\Ctil (tu+(1-t)\util)(u-\util) \\
&\leq 
\Ctil u_+ (u-\util) \\
&= 
\Ctil u_0 v_+ (u_0 v-u_0 \vtil) \\
&= 
\Ctil u_0^2 v_+(v- \vtil),
\end{split}\]
where $0 \leq t \leq 1$, $v_+ = u_0^{-1} u_+$, and the constant 
$\Ctil > 0$ depends only on $f$. Consequently, we can estimate 
\[\begin{split}
\int_{S^{n-1}} (v(s,\theta)-\vtil (s,\theta)) \, d\mu^{\sigma} 
&\leq 
\frac{(n-2)\omega_{n-1}}{2} \left( m^g - m^{\gtil} \right) \\ 
&\qquad + 
\int_s^{\infty} F(r) \left( 
\int_{S^{n-1}} (v(r,\theta)- \vtil(r,\theta) ) \, d\mu^{\sigma} 
\right) \, dr,
\end{split}\]
where $F(r) \definedas C' \cosh r \sinh^{n-1} r\, u_0^2(r) v_+(r)$. 

We now argue as in the proof of Gronwall's lemma and prove the estimate 
\begin{equation} \label{eqGronwall}
\int_{S^{n-1}}(v(s,\theta)-\vtil (s,\theta)) \, d\mu^{\sigma}
\leq \frac{(n-2)\omega_{n-1}}{2}
\left( m^g - m^{\gtil} \right) 
e^{\int_s^{\infty}F(t)\,dt}.
\end{equation}
We first consider the case when $m^g - m^{\gtil} > 0$ and set 
\[
G(s) \definedas 
\frac{(n-2)\omega_{n-1}}{2} \left( m^g - m^{\gtil} \right) 
+ \int_s^{\infty}F(r)\left(\int_{S^{n-1}}(v(r,\theta)- \vtil(r,\theta))
\, d\mu^{\sigma} \right) \, dr.
\]
Thus we have 
$\int_{S^n}(v(s,\theta)-\vtil (s,\theta)) \, d\mu^{\sigma}\leq G(s)$, and 
$G(s)\geq \frac{(n-2)\omega_{n-1}}{2}\left(m^g - m^{\gtil} \right)$. It is 
also clear that
\[
G'(s) 
= -F(s)\int_{S^{n-1}} (v(s,\theta)-\vtil (s,\theta)) \, d\mu^{\sigma}
\geq -F(s)G(s).
\]
Since $G(s)>0$ we conclude that
\[
\frac{G'(s)}{G(s)} \geq -F(s).
\] 
Integrating this inequality from $s$ to $\infty$ we get 
\[
\ln \left( \frac{(n-2)\omega_{n-1}}{2} 
\left(m^g - m^{\gtil} \right) \right) 
- \ln G(s) \geq -\int_s^{\infty} F(t)\,dt.
\] 
This yields 
\[
G(s) 
\leq \frac{(n-2)\omega_{n-1}}{2}
\left(m^g - m^{\gtil} \right) e^{\int_s^{\infty}F(t)\,dt},
\]
which in its turn implies \eqref{eqGronwall}. Note that 
\eqref{eqGronwall} also holds for $m^g - m^{\gtil} = 0$ which 
follows by passing to the limit when $m^g - m^{\gtil} > 0$ and 
$m^g - m^{\gtil} \rightarrow 0$ in \eqref{eqGronwall}.

As a consequence we can estimate the $L^p$-norm of $v - \vtil$ 
over the annulus $A_{r_1,r_2}$ where $R_0 < r_1 < R_1 < r_2$. We have
\[\begin{split}
\| v-\vtil \|^p_{L^p(A_{r_1,r_2})}
&=
\int_{A_{r_1,r_2}} (v-\vtil)^p \, d\mu^{b} \\
&\leq 
\int_{A_{r_1,r_2}} (2v_+)^{p-1} (v-\vtil) \, d\mu^{b} \\
&=
\int_{r_1}^{r_2} 
(2v_+)^{p-1} \sinh^{n-1} r 
\left( \int_{S^{n-1}} (v(r,\theta)-\vtil(r,\theta)) \, d\mu^{\sigma} \right)
\, dr\\
&\leq C \left( m^g - m^{\gtil} \right)
\end{split}\]
for some positive constant $C$.

We are now about to obtain the estimate stated in the lemma. The 
equation for $U - \Util$ reads 
\[
-\frac{4(n-1)}{n-2}\Delta^b (U-\Util)-n(n-1)\left(U-\Util\right)
=
-n(n-1) \left(U^{\kappa+1}-\Util^{\kappa+1}\right).
\]
Since $u_0$ is bounded we have  
\[
\|U-\Util\|_{L^p(A_{r_1,r_2})}
= 
\|u-\util\|_{L^p(A_{r_1,r_2})}
\leq 
C \left( m^g - m^{\gtil} \right)^{1/p}. 
\]
Here and in the rest of the proof the value of the positive constant 
$C$ might vary from line to line but remains independent of 
$(M,g,\Phi,U)\in \cA$. By the mean value theorem we have 
\[\begin{split}
U^{\kappa+1} - \Util^{\kappa+1}
&=
(\kappa+1) \left(tU+(1-t)\Util\right)^{\kappa} \left(U-\Util\right) \\
&\leq 
C U_+^{\kappa} \left( U - \Util \right) \\
&\leq C \left(U-\Util\right)
\end{split}\]
on $A_{r_1,r_2}$ for some $t \in [0,1]$. Hence
\[
\| U^{\kappa+1} - \Util^{\kappa+1} \|_{L^p(A_{r_1,r_2})}
\leq C \left(m^g - m^{\gtil}\right)^{1/p}. 
\]
Now elliptic regularity yields
\[
\left\|U-\Util\right\|_{W^{2,p}(A_{r'_1,R_1})}
\leq C \left(m^g - m^{\gtil}\right)^{1/p},
\]
where $r_1 < r'_1 < R_1$, and by embedding theorems we conclude that 
\[
\sup_{A_{r'_1,R_1}}\left|U-\Util\right|
\leq C \left(m^g - m^{\gtil}\right)^{1/p}.
\]

Set $\phi \definedas \log U$ and $\phitil \definedas \log \Util$. 
Then $\phi - \phitil$ is non-negative, tends to zero at infinity, 
and satisfies 
\[
-\frac{4(n-1)}{n-2}\left(\Delta^b (\phi-\phitil) 
+ \langle d(\phi-\phitil),d(\phi+\phitil)\rangle_b\right)
+n(n-1)\left(e^{\kappa \phi}-e^{\kappa \phitil}\right) = 0.
\]
If the maximum of $\phi - \phitil$ is attained at an interior point
of $\bH^n \setminus B_{R_1}$ we get a contradiction, and thus
\[
\log U - \log \Util 
\leq 
\sup_{\partial B_{R_1}} \left( \log U - \log \Util \right)
\]
on $\bH^n \setminus B_{R_1}$. By the mean value theorem we have
\[
\log U - \log \Util 
= 
\frac{U-\Util}{tU+(1-t)\Util}
\begin{dcases}
\geq \frac{U-\Util}{U_+(R_1)} \\
\leq 
\frac{U-\Util}{U_-(R_1)}
\end{dcases}
\]
for some $t \in [0,1]$. Thus 
\[\begin{split}
U - \Util
&\leq
U_+(R_1) \left( \log U - \log \Util \right) \\
&\leq
U_+(R_1) \sup_{\partial B_{R_1}} \left( \log U - \log \Util \right) \\
&\leq
\frac{U_+(R_1)}{U_-(R_1)}
\sup_{\partial B_{R_1}} \left( U - \Util \right) \\
&\leq
C \left(m^g - m^{\gtil}\right)^{1/p}
\end{split}\]
on $\bH^n \setminus B_{R_1}$, which concludes the proof of the proposition.
\end{proof}

\subsection{Proof of lemmas}

We now complete the proof of Theorem \ref{thmNearPMT1} by proving the 
lemmas stated in Subsection \ref{secStrategy}.

\begin{proof}[Proof of Lemma \ref{lmBoundV}]
We first construct $V_\pm$. The construction being lengthy, we give only
the argument for $V_+$. We want $V_+$ to be a supersolution for Equation
\eqref{eqEigenVect},
\[
-\Delta^g V_+ + n V_+ \geq 0.
\]
Since $g = U^\kappa b$ on $\bH^n \setminus B_{R_0}(0)$ the previous
inequality is equivalent to
\[
- \Delta^b V_+ 
- 2 \left\< \frac{dU}{U}, dV_+\right\> 
+ n U^\kappa V_+ 
\geq 0.
\]
We choose $V_+$ to be a function of $r$ so
\[
- V_+'' - (n-1) \coth r V_+' 
- 2 \frac{\partial_r U}{U} V_+' 
+ n U^\kappa V_+ 
\geq 0,
\]
where a prime denotes a derivative with respect to $r$. From Lemma 
\ref{lmUniformBound}, there exists a universal constant $A'_1$ depending 
only on $R'_0$ such that $\frac{\partial_r U}{U} \leq A'_1 e^{-n r}$. 
Assuming that $V_+, V_+' \geq 0$, the previous inequality will be 
satisfied provided that
\begin{equation}\label{eqSurSolution}
-V_+'' - (n-1) \coth r V_+' - 2 A'_1 e^{-n r} V_+' 
+ n \phi_-^\kappa V_+ = 0, 
\end{equation}
where $\phi_-$ is the anti-de~Sitter-Schwarzschild solution vanishing 
at $r = R_0$. Let $\lambda$ be a positive real number to be chosen later. 
From standard theory, there exists a unique solution to Equation 
\eqref{eqSurSolution} defined on $[R'_0, \infty)$ such that $V_+(R'_0) 
= \lambda$ and $V_+'(R'_0) = 0$.

We first claim that $V_+$ and $V'_+$ are both positive functions on 
$(R'_0, \infty)$. Indeed, rewriting Equation \eqref{eqSurSolution} as
\begin{equation}\label{eqSurSolution2}
V_+'' + \left( (n-1) \coth r + 2 A'_1 e^{-n r} \right) V_+' 
= 
n \phi_-^\kappa V_+,
\end{equation}
setting $R \definedas \inf \{r > R'_0, V_+(r) \leq 0\}$, and assuming 
that $R < \infty$, we have $V_+ > 0$ on $(R'_0, R)$ and $V_+(R) = 0$. 
Hence, regarding \eqref{eqSurSolution2} as a first order homogeneous 
ordinary differential equation for $V_+'$, we conclude that $V_+' > 0$ 
on $(R'_0, R)$. In particular, $V_+(R) \geq V_+(R'_0) = \lambda > 0$. 
This contradicts the definition of $R$. The claim is proved.

Next we prove that $V_+ = \alpha \lambda \cosh r + O(e^{-(n-1)r})$ for 
some constant $\alpha > 0$. Hence setting $\lambda = 1/\alpha$, 
we get a supersolution to Equation \eqref{eqEigenVect} such that 
$V_+ \sim \cosh r = V_{(0)}$. To prove this second claim we set 
$V_+(r) \definedas \cosh r v_+(r)$. By a straightforward calculation, 
we find that $V_+$ satisfies \eqref{eqSurSolution} if and only if $v_+$ 
satisfies 
\[\begin{split}
& v_+'' + \left(2 \tanh r + (n-1) \coth r + 2 A'_1 e^{-n r}\right) v_+' \\
&\qquad \qquad + 
\left(2 A'_1 e^{-n r} \tanh r + n \left(\phi_-^\kappa - 1\right)\right) v_+ 
= 0.
\end{split}\]
From the first claim we have $v_+ > 0$. We introduce 
$k_+ \definedas \frac{v_+'}{v_+}$ and obtain the following Riccati 
equation for $k_+$,
\begin{equation}\label{eqSurSolution4}
\begin{split}
&k_+' + k_+^2 + \left(2 \tanh r + (n-1) \coth r 
+ 2 A'_1 e^{-n r}\right) k_+ \\
&\qquad \qquad + 
\left(2 A'_1 e^{-n r} \tanh r 
+ n \left(\phi_-^\kappa - 1\right)\right) = 0.
\end{split}
\end{equation}
Without loss of generality, we can assume that $A'_1$ is chosen so 
large that
\[
2 A'_1 e^{-n r} \tanh r + n \left(\phi_-^\kappa - 1\right) \geq 0
\]
on $(R'_0, \infty)$. From the boundary condition $V_+'(R'_0) = 0$ we
have 
\[
v_+'(R'_0) = - \tanh R'_0 v_+(R'_0) < 0.
\]
It is then fairly straightforward to argue that $-1 < k_+ < 0$ on 
$(R'_0, \infty)$. For this let $R$ be the smallest $r > R'_0$ such that 
$k_+(r) \geq 0$. Then $k_+(R) = 0$ and, from Equation 
\eqref{eqSurSolution4}, $k_+'(R) < 0$ so $k_+ (r) > 0$ for some $r$ 
slightly smaller than $R$, contradicting the definition of $R$. This 
estimate can be further refined. We select 
$\alpha \in \left(\frac{n}{2}, n\right)$ and set 
$k_+^- \definedas - e^{-\alpha(r-r_0)}$ for some $r_0$ to be chosen later. 
Then $k_+^- \geq -1$ on the interval $[r_0, \infty)$. Hence
\[\begin{split}
&(k_+^-)' + (k_+^-)^2 
+ \left(2 \tanh r + (n-1) \coth r + 2 A'_1 e^{-n r}\right) k_+^-\\
&\qquad 
= \left(\alpha + e^{-\alpha (r-r_0)} - 2 \tanh r - (n-1) \coth r 
- 2 A'_1 e^{-n r}\right) e^{-\alpha(r-r_0)}\\
&\qquad 
\leq \left(\alpha - n\right) e^{-\alpha(r-r_0)},
\end{split}\]
where we used the inequality
\[\begin{split}
2\tanh r + (n-1) \coth r
&= 
2\left(\frac{1}{\coth r} + \coth r \right) + (n-3) \coth r \\
&\geq 
2 + (n-3) \coth r \\
&\geq 
n-1.
\end{split}\]
Consequently, choosing $r_0$ large enough, we can ensure that
\[\begin{split}
&(k_+^-)' + (k_+^-)^2 
+ \left(2 \tanh r + (n-1) \coth r + 2 A'_1 e^{-n r}\right) k_+^- \\
&\qquad \qquad + 
\left(2 A'_1 e^{-n r} \tanh r + n \left(\phi_-^\kappa - 1\right)\right) < 0
\end{split}\]
on the interval $[r_0, \infty)$. A slight modification of the previous
argument shows that $k_+^- \leq k_+ \leq 0$. Equation \eqref{eqSurSolution4}
then implies $k_+' = O(e^{-n r})$. Together with the fact that $k_+ \to 0$
at infinity, this implies $k_+ = O(e^{-n r})$. Thus we infer that
\[
\log v_+(r) = \log \lambda + \mu + O(e^{-nr})
\] 
for some constant $\mu$. Hence,
\[
v_+(r) = \lambda e^\mu + O(e^{-n r}).
\] 
This proves the second claim with $\alpha = e^{\mu}$.

Finally remark that since $k_+(r) = \frac{v'_+(r)}{v_+(r)} \leq 0$ and 
$v_+ \to 1$ at infinity, it follows that $v_+(r) \geq 1$ so $V_+ \geq V_{(0)}$.

The construction of the subsolution $V_-$ on $\bH^n \setminus B_{R'_0}$ 
is entirely similar. The only difference is that we select 
$V_-(R'_0) = 0$ and $V_-'(R'_0) > 0$. The function $V_-$ then satisfies 
$V_- \leq V_{(0)}$.

From now on we will work on the entire manifold $M$. Using the 
diffeomorphism $\Phi$ we define open sets $B'_R$ in $M$ through the 
relation $\Phi( M \setminus B'_R ) = \bH^n \setminus B_{R}$ for 
$R \geq R_0$. The set $B'_R$ is the part of $M$ inside an approximate
geodesic sphere in the asymptotically hyperbolic end. By abusing 
notation we consider the functions $V_{\pm}$ and $V_{(0)}$ as defined 
on $M \setminus K$ through the diffeomorphism $\Phi$. 

Our proof of existence of the function $V^g$ follows \cite{GrahamLee}. 
For any $r > R'_0$ there exists a unique function $V^r$ solving 
\eqref{eqEigenVect} inside the sphere of radius $r$ with Dirichlet 
data $V = V_{(0)} $ on $\partial B'_{r}$. From the maximum 
principle, $V^r \geq 0$. Then a second application of the maximum 
principle in the annulus $\overline{B'}_{r} \setminus B'_{R'_0}$ 
yields $V^r \geq V_-$. We extend the function $V_+$ by $\lambda$
on $B'_{R'_0}$. This new function $V_+$ is a $C^1$-supersolution of 
\eqref{eqEigenVect} in the weak sense. Hence $V^r \leq V_+$ (see 
for example \cite[Theorem 8.1]{GilbargTrudinger} for more details). 
In particular, the functions $V^r$ are uniformly bounded on compact 
subsets. Then a standard argument using elliptic regularity and a 
diagonal extraction process yields the existence of the function $V^g$. 
Similarly, we extend the function $V_-$ by zero on $B'_{R'_0}$. 
The function $V_-$ extended this way becomes a subsolution in the 
weak sense so the functions $V^r$ satisfy $V^r \geq V_-$. In the 
limit, the function $V^g$ is pinched between $V_-$ and $V_+$, that is
\[
V_- \leq V^g \leq V_+.
\] 
This proves that $V^g - V_{(0)} = O(e^{-(n-1)r})$.

We note that
\[
\Delta^g V_{(0)}
= U^{-\kappa} \left(n \cosh r + 2 \frac{\partial_r U}{U} \sinh r \right)
= n V_{(0)} + O(e^{-(n-1)r}).
\]
Hence,
\[
\left(- \Delta^g + n\right) \left(V^g - V_{(0)}\right)
= O(e^{-(n-1)r}).
\]
The estimates for $d(V^g -V_{(0)})$ and $\left|\nabla^{(k)} V^g\right|$
follow from standard elliptic regularity.

We finally prove uniqueness of $V^g$. Assume that $V_1$ is the function
we constructed before so that $V_- \leq V_1 \leq V_+$ and $V_2$ is another
function satisfying $\Delta V_2 = n V_2$, $V_2 \sim V_{(0)}$. From the
strong maximum principle we have $V_1 > 0$. We compute
\[\begin{split}
n V_2 
&= 
\Delta \left(\frac{V_2}{V_1} V_1\right)\\
&= 
V_1 \Delta \frac{V_2}{V_1} 
+ 2 \left\< d V_1, d\left(\frac{V_2}{V_1}\right)\right\> 
+ \frac{V_2}{V_1} \Delta V_1\\
&= 
V_1 \Delta \frac{V_2}{V_1} 
+ 2 \left\< d V_1, d\left(\frac{V_2}{V_1}\right)\right\> + n V_2, \\
\end{split} \]
so 
\[
0 = \Delta \frac{V_2}{V_1} 
+ 2 \left\< \frac{d V_1}{V_1}, d\left(\frac{V_2}{V_1}\right)\right\>.
\]
Since $V_1 \sim V_2$, the function $V_2/V_1$ tends to $1$ at infinity.
From the strong maximum principle (which can be applied here since if
$V_2/V_1$ is not constant, the maximum of $|V_2/V_1-1|$ is attained at 
some point $p \in M$), we conclude that $V_2/V_1 = 1$.
\end{proof}

\begin{proof}[Proof of Lemma \ref{lmConformalFactor}]
From Lemmas \ref{lmUniformBound} and \ref{lmBoundV}, there are
universal constants $B_0, B_1, \ldots$ such that 
\[
\left|\nabla^{(k)} T\right| \leq B_k
\]
for $k = 0, 1, \ldots$ on the support of $\chi$. Hence
\[
\left| g_s(X,X) - g(X,X) \right| 
= \left| s T(X, X)\right| \leq |s| B_0 g(X, X)
\]
for any $X \in TM$. So if $|s| \leq \frac{1}{2 B_0}$ we have
\[
\frac{1}{2} g(X,X) \leq g_s(X,X) \leq \frac{3}{2} g(X,X).
\]
We denote by $\nabla^{g_s}$ the Levi-Civita connection of $g_s$. The
difference between $\nabla^{g_s}$ and $\nabla^{g_0}$ is a symmetric vector
valued 2-tensor $\Gamma(s)$,
\[
\nabla^{g_s}_X Y - \nabla^{g_0}_X Y = \Gamma(s)(X, Y).
\]
In coordinates $\Gamma(s)$ is given by
\[ \begin{split}
\Gamma^k_{ij}(s)
&= \frac{1}{2} g_s^{kl} 
\left( 
\nabla_i (g_s)_{lj} + \nabla_j (g_s)_{il} - \nabla_l (g_s)_{ij}
\right)\\
&= 
\frac{s}{2} g_s^{kl} 
\left(
\nabla_i (\chi T_{lj}) + \nabla_j (\chi T_{il}) - \nabla_l (\chi T_{ij})
\right),
\end{split} \]
where we have denoted by $\nabla = \nabla^{g_0}$ the Levi-Civita 
connection of the metric $g_0 = g$. The scalar curvature of the metric 
$g_s$ can be written as follows,
\[
\scal^{g_s}
= g_s^{ij} \ric^{g_0}_{ij}
+ g_s^{jl} \left(\nabla_i \Gamma_{jl}^i(s) - \nabla_l \Gamma_{ij}^i(s)
+ \Gamma_{ip}^i(s) \Gamma_{jl}^p(s) - \Gamma_{lp}^i(s) \Gamma_{ij}^p(s) \right).
\]
From this formula it is not complicated to see that there is a constant 
$s_0 > 0$, $s_0 \leq \frac{1}{2 B_0}$, depending only on $B_0$, $B_1$, 
$B_2$ and $n$ such that
\[
\left|\scal^{g_s} - \scal^g\right| \leq n-1
\]
for $|s| \leq s_0$. From the bound on $\scal^{g_s}$ it follows that 
the constant functions 
$\phi_- = \left(\frac{n-1}{n}\right)^{1/\kappa}$ and
$\phi_+ =\left(\frac{n+1}{n}\right)^{1/\kappa}$ are
respectively a sub-solution and a super-solution of the Yamabe equation
\begin{equation}\label{eqYamabegs}
-\frac{4(n-1)}{n-2} \Delta^{g_s} \phi_s + \scal^{g_s} \phi_s 
+ n(n-1) \phi_s^{\kappa+1}=0.
\end{equation}
Arguing as in the proof of Proposition \ref{propAtoA0} there exists 
a unique solution $\phi_s$ of \eqref{eqYamabegs} such that $\phi_s$ is 
bounded from above and away from zero. Further 
$\phi_- \leq \phi_s \leq \phi_+$.

We next prove that the map $s \mapsto \phi_s$ is $C^2$. We consider 
the map
\[\begin{array}{rccl}
\Xi : & \Omega \times [-s_0, s_0] & \to & X^{0, p}_\delta\\
& (u, s) & \mapsto & -\frac{4(n-1)}{n-2} \Delta^{g_s} u 
+ \scal^{g_s} (u+1) + n (n-1) (u+1)^{\kappa+1},
\end{array}\]
where $\Omega = \{u \in X^{2, p}_\delta, u > -1\}$. Hence, for any
$s \in [-s_0, s_0]$, $u_s = \phi_s - 1$ is the only solution to the 
equation $\Xi(u, s) = 0$. Further $\Xi$ is a $C^2$ function.
The differential of $\Xi$ with respect to $u$ at any point $(u_s, s)$ 
is given by
\[\begin{array}{rccl}
D_u \Xi(u_s, s) : & X^{2, p}_\delta & \to & X^{0, p}_\delta\\
& v & \mapsto & -\frac{4(n-1)}{n-2} \Delta^{g_s} v 
+ \left(\scal^{g_s} + (\kappa+1) n (n-1) \phi_s^\kappa\right) v.
\end{array}\]
We remark that
\[\begin{split}
\scal^{g_s} + (\kappa+1) n (n-1) \phi_s^\kappa
&\geq 
-(n+1)(n-1) + (\kappa+1) (n-1)^2\\
&\geq 
\frac{2n(n-1)}{n-2},
\end{split}\]
from which it follows that the $L^2$-kernel of $D_u \Xi(u_s, s)$ is zero. 
From the Fredholm alternative (see 
\cite[Proof of Proposition 5.1]{GicquaudSakovich}), we conclude that
$D_u \Xi(u_s, s)$ is invertible. Using the implicit function theorem, 
this proves that the map $s \mapsto \phi_s - 1 \in X^{2, p}_\delta$ is 
$C^2$.

To prove the asymptotics of $\phi_s$, remark that the metric $\lambda_s$ 
falls into the class $\cA(R_1'')$. Hence the estimates \eqref{eqEstimatesPhi} 
are consequences of Lemma \ref{lmUniformBound}.
\end{proof}

\begin{proof}[Proof of Lemma \ref{lmMassLambda}]
We first estimate the first and second derivatives of $\phi_s$ with respect 
to $s$. We differentiate Equation \eqref{eqYamabegs} with respect to $s$ 
and find the following equation for $\dot{\phi}_s$,
\[\begin{split}
&-\frac{4(n-1)}{n-2} \Delta^{g_s} \dot{\phi}_s + \scal^{g_s} \dot{\phi}_s 
+ (\kappa+1)n(n-1) \phi_s^{\kappa} \dot{\phi}_s\\
&\qquad 
= \frac{4(n-1)}{n-2} \pdiff{\Delta^{g_s}}{s} \phi_s 
- \pdiff{\scal^{g_s}}{s} \phi_s.
\end{split}\]
Note that the right hand side has support in the annulus $A_{R'_0, R_1}$. Thus,
by Lemma \ref{lmConformalFactor}, it is bounded by some universal constant
$C$. We also remark that, since 
\[
\scal^{g_s} + (\kappa+1) n (n-1) \phi_s^\kappa > \frac{2n(n-1)}{n-2}
\] 
and since $\dot{\phi_s}$ tends to zero at infinity (this is a consequence 
of $\dot{\phi_s} \in X^{2, p}_\delta$), we have
\[
\sup |\dot{\phi_s}| \leq \frac{n-2}{2n(n-1)} C.
\]
By standard techniques one can then prove that 
$\left\|\dot{\phi_s}\right\|_{X^{2, p}(\bH^n \setminus B_{R'_0})} \leq C$ for
some universal constant $C$. The same strategy can then be used to study the
second order derivative of $\phi_s$. However, the calculations are lengthy
and we do not include the argument here.

The last step is to prove that $H(s)$ is a $C^2$ function of $s$.
For this we write $H(s)$ as follows (see 
\cite[page 114]{HerzlichMassFormulae} or \cite{ChruscielHerzlich} for 
more details),
\[\begin{split}
H(s)
&= 
H_\Phi^{\lambda_s} (V)\\
&= 
\int_{S_{R_2}} \left(
V (\divg^b e_s - d \tr^b e_s) + (\tr^b e_s) dV - e_s(\nabla^b V, \cdot)
\right) (\nu_{R_2}) \, d \mu^b\\
&\qquad 
+ \int_{\bH^n \setminus B_{R_2}} 
\left(
V \left(\scal^{\lambda_s} - \scal^b\right) + \mathcal{Q}(e_s, V)
\right) 
\, d \mu^b\\
&= 
\int_{S_{R_2}} \left(
V (\divg^b e_s - d \tr^b e_s) + (\tr^b e_s) dV - e_s(\nabla^b V, \cdot)
\right) (\nu_{R_2}) \, d \mu^b\\
&\qquad + \int_{\bH^n \setminus B_{R_2}} \mathcal{Q}(e_s, V) \, d \mu^b
\end{split}\]
where $e_s = \lambda_s - b = (\phi_s^\kappa U^\kappa - 1) b$, and
$\mathcal{Q}(V, e)$ is an expression which is linear in $V$, quadratic 
in $e_s$ and its first derivatives, and cubic in $(\lambda_s)^{-1}$. 
It corresponds to the negative of the non-linear terms in the Taylor 
expansion of
\[
\int_{\bH^n \setminus B_{R_2}} V \left(\scal^{\lambda_s} - \scal^b\right) \, d \mu^b
\]
with respect to $e_s = \lambda_s - b$. Since $\lambda_s = (U \phi_s)^\kappa b$
on $\bH^n \setminus B_{R_2}$, this expression can be explicitly computed,
\[
\mathcal{Q}(e_s, V) =
 V \left(
(n-1)\left(\frac{1}{\psi_s^2}-1\right) \Delta^b \psi_s
 + n (n-1) \frac{(\psi_s-1)^2}{\psi_s}
 + \frac{(n-1)(n-6)}{4 \psi_s} \left|\frac{d\psi_s}{\psi_s}\right|^2_b 
\right),
\]
where $\psi_s \definedas (U \phi_s)^\kappa$. Written in this form, one can 
conclude from standard theorems on differentiation of integrals
that $H(s)$ depends on $s$ in a $C^2$ fashion.

From the estimates we have found for $\dot{\phi}(s)$ and $\ddot{\phi}(s)$
together with Lemmas \ref{lmUniformBound} and \ref{lmConformalFactor} it 
is not complicated to deduce that $\dot{H}(s)$ and $\ddot{H}(s)$ are
uniformly bounded on the interval $[-s_0, s_0]$.
\end{proof}

\begin{proof}[Proof of Lemma \ref{lmVariationMass}]
For $\lambda_s$ we have 
\[
e_s 
= \lambda_s - b 
= \phi_s^{\kappa} 
\left(g + s \chi \left(\tlric^g - \frac{\tlhess V^g}{V^g}\right)\right) - b,
\]
the derivative of this with respect to $s$ evaluated at $s=0$ is 
\[
\dot{e} = 
\kappa \dot{\phi} g + \chi \left(\tlric^g - \frac{\tlhess V^g}{V^g}\right) 
=: e^1 + e^2 .
\]
The conformal factors $\phi_s$ satisfy the Yamabe equation 
\[
-\frac{4(n-1)}{n-2} \Delta^{g_s} \phi_s + \scal^{g_s} \phi_s
= -n (n-1) \phi_s^{\kappa+1}.
\]
Differentiating this at $s = 0$ and using the fact that $\phi_0 = 1$ 
we find that
\[
-\frac{4(n-1)}{n-2} \Delta^g \dot{\phi}
+ \dot{\scal}^g(\dot{g}) + \scal^g \dot{\phi}
= -n(n-1)\frac{n+2}{n-2} \dot{\phi}, 
\]
or
\begin{equation} \label{phidot}
\frac{4(n-1)}{n-2}
\left( \Delta^g \dot{\phi} - n \dot {\phi} \right)
= \dot{\scal}{}^g(\dot{g}).
\end{equation}

We compute
\begin{equation} \label{mdot}
\begin{split}
\dot{H} (0) 
&= 
\frac{d}{ds} \left( H_{\Phi}^{\lambda_s} (V) \right)_{|s=0} \\
&=
\lim_{r \to \infty} \int_{S_r} \left(
V (\divg^b \dot{e}- d \tr^b \dot{e}) 
+ (\tr^b \dot{e}) dV - \dot{e}(\nabla^b V, \cdot)
\right) (\nu_r) \, d \mu^b \\
&=
\lim_{r \to \infty} \int_{S_r} \left(
V^g (\divg^g \dot{e}- d \tr^g \dot{e}) 
+ (\tr^g \dot{e}) dV^g - \dot{e}(\nabla^g V^g, \cdot)
\right) (\nu_r^g) \, d \mu^g,
\end{split}
\end{equation}
where we can change from the metric $b$ to the metric $g$ since 
$g$ is asymptotically hyperbolic and the function $V^g$ has the 
asymptotics specified in Lemma \ref{lmBoundV}. Note that since 
$e_2$ has compact support its contribution to $\dot{H}(0)$ is zero. 
For the terms with $e^1 = \kappa \dot{\phi} g$ in \eqref{mdot} we 
have
\[\begin{split}
\dot{H}(0)
&= 
\lim_{r \to \infty} \int_{S_r} \left(
V^g (\divg^g e^1- d \tr^g e^1) 
+ (\tr^g e^1) dV^g - e^1(\nabla^g V^g, \cdot)
\right) (\nu_r^g) \, d \mu^g \\
&=
\lim_{r \to \infty} \frac{4(n-1)}{(n-2)} \int_{S_r} 
\left( \dot{\phi} dV^g - V^g d \dot{\phi} \right) 
(\nu_r^g) \, d \mu^g \\
&=
\frac{4(n-1)}{(n-2)} \int_{M} 
\divg^g \left( \dot{\phi} dV^g - V^g d \dot{\phi} \right) 
\, d \mu^g \\
&=
\frac{4(n-1)}{(n-2)} \int_{M} 
\left(\dot{\phi} \Delta^g V^g - V^g \Delta^g \dot{\phi} \right) 
\, d \mu^g \\
&=
\frac{4(n-1)}{(n-2)} \int_{M} V^g
\left(n \dot{\phi} - \Delta^g \dot{\phi} \right) 
\, d \mu^g \\
&=
- \int_{M} V^g \dot{\scal}{}^g(\dot{g}) \, d \mu^g ,
\end{split}\]
where the last equality was obtained using \eqref{phidot}. Here
$\dot{g} = \chi \left(\tlric^g - \frac{\tlhess V^g}{V^g}\right)$
is traceless. So from the formula for the first variation of scalar
curvature, see \cite[Theorem~1.174]{Besse}, we obtain
\[\begin{split}
\dot{\scal}{}^g(\dot{g})
&= 
\divg^g \divg^g \dot{g} - \Delta^g \tr^g \dot{g} - \<\dot{g}, \ric^g\>_g\\
&= 
\divg^g \divg^g \dot{g} - \<\dot{g}, \tlric^g\>_g\\
&= 
- \chi \left| \tlric^g \right|_g^2 
+ \frac{\chi}{V^g} \left\< \tlhess V^g, \tlric^g \right\>
+ \divg^g \divg^g (\chi \tlric^g) 
- \divg^g \divg^g \frac{\chi \tlhess V^g}{V^g}.
\end{split}\]
Thus, replacing this expression in the formula for $\dot{H}(0)$ and 
integrating by parts, we get
\[\begin{split}
\dot{H}(0)
&= 
\int_M \chi V^g \left| \tlric^g \right|_g^2 \, d\mu^g
- 2 \int_M \chi \left\< \tlric^g, \tlhess V^g \right\>_g \, d\mu^g\\
&\qquad 
+ \int_M \chi \frac{1}{V^g} \left| \tlhess V^g \right|_g^2 \, d\mu^g\\
&= 
\int_M \chi V^g 
\left| \tlric^g - \frac{1}{V^g} \tlhess V^g\right|_g^2 \, d\mu^g.
\end{split}\]
\end{proof}

\section{An alternative argument for spin manifolds}
\label{secSpincase}


In this section we will prove a version of Theorem \ref{thmNearPMT1} 
with an argument using spinors. This follows closely the ideas of  
\cite[Section~12]{BrayPenrose}, see also the appendix of 
\cite{LeeNearEquality}. We only give a sketch of the argument. We first 
introduce the following class of asymptotically hyperbolic manifolds. 

\begin{definition} 
For $R_0 > 0$, we define the class $\cA^{\spin}(R_0)$ of 4-tuples
$(M, g, \Phi, U)$ such that
\begin{itemize}
\item 
$(M, g)$ is a complete Riemannian spin manifold which is asymptotically 
hyperbolic with respect to $\Phi$, where $\Phi$ is a diffeomorphism from 
the exterior of a compact set $K \subset M$ to 
$\bH^n \setminus \overline{B}_{R_0}$;
\item 
$\scal^g \geq -n(n-1)$, and $\scal^g = -n(n-1)$ on $M \setminus K$;
\item
$U$ is a positive function on $\bH^n \setminus B_{R_0}$ such that 
$U \to 1$ at infinity and $\Phi_* g = U^\kappa b$;
\item
the coordinates at infinity $\Phi$ are balanced.
\end{itemize}
\end{definition}

We prove the following theorem on the near-equality case of the 
positive mass theorem for spin manifolds.

\begin{maintheorem} \label{thmNearPMT2}
Let $R_1 > R_0$ and $\epsilon > 0$. There is a constant $\delta > 0$
so that  
\[
|U-1| \leq \epsilon e^{-nr}
\]
on $\bH^n \setminus B_{R_1}$ for all 
$(M, g, \Phi, U) \in \cA^{\spin}(R_0)$ with $m^g < \delta$.
\end{maintheorem}

We fix the constant $R_0 > 0$ and abbreviate $\cA^{\spin}(R_0) = \cA^{\spin}$.
We begin by describing the relationship between Killing spinors and 
the asymptotically hyperbolic mass, for this we follow closely the 
discussion in \cite[Section~4]{ChruscielHerzlich}.

Since $M$ is a spin manifold there is a spin structure and an associated
spinor bundle $\Sigma M$ on $(M,g)$. On $\Sigma M$ we define the 
connection $\widehat{\nabla}^g$ by 
\[
\widehat{\nabla}^g_X \phi \definedas 
\nabla^g_X \phi + \frac{i}{2} X \cdot \phi.
\]
Here  $\nabla^g$ is the Levi-Civita connection for the metric $g$, $\phi$ 
is a section of the spinor bundle, and the dot denotes the Clifford action 
of tangent vectors on spinors. Spinors $\phi$ which are parallel with 
respect to $\widehat{\nabla}^g$ are called (imaginary) Killing spinors.

We will now describe the Killing spinors on hyperbolic space. The ball 
model of hyperbolic space is given by the metric 
$\omega^{-2} \xi$ where $\omega(x) = \frac{1}{2}(1 - |x|^2)$ and $\xi$ is 
the flat metric on the open unit ball $B^n$ in $\bR^n$. In this model 
the Killing spinors on $\bH^n$ are all spinors of the form
\[
\phi_s(x) = \omega(x)^{-1/2} (1 - i x \cdot) s.
\]
Here $s$ is a constant spinor on $(B^n,\xi)$, or equivalently an element
of the spinor representation space $\Sigma$. For the Clifford action 
we identify points in $B^n$ with tangent vectors. For any Killing spinor 
$\phi_s$ on $\bH^n$ its squared norm $V_s \definedas |\phi_s|^2$ is an 
element of $\cN$. Every element of $\cN$ of the form 
$V_{(0)} - \sum_{i=1}^n a_i V_{(i)}$ where $(a_1, \dots, a_n) \in S^{n-1}$ 
is equal to $V_s$ for some Killing spinor $\phi_s$.

Using the connection $\widehat{\nabla}^g$ we define the Dirac operator
$\widehat{D}^g$ by 
\[
\widehat{D}^g \phi \definedas 
\sum_{i=1}^n e_i \cdot \widehat{\nabla}^g_{e_i} \phi
= D^g \phi - \frac{in}{2} \phi, 
\]
where $e_i$, $i=1,\dots,n$, is an orthonormal frame for $g$ and 
$D^g = \sum_{i=1}^n e_i \cdot \nabla^g_{e_i}$ is the Dirac operator associated
to $\nabla^g$. The Schr\"odinger-Lichnerowicz formula for $\widehat{D}^g$
has a boundary term related to the asymptotically hyperbolic mass. 
If $(M,g)$ is an asymptotically hyperbolic manifold with diffeomorphism 
$\Phi : M \setminus K \to \bH^n \setminus B$ at infinity, 
then the Killing spinor $\phi_s$ on $\bH^n$ can be pulled back to a 
spinor $\Phi^* \phi_s$ on $M \setminus K$. If $\psi_s$ is a spinor on $M$
with $\widehat{D}^g \psi_s = 0$ and $\psi_s - \Phi^*\phi_s \to 0$ at 
infinity then the Schr\"odinger-Lichnerowicz formula for $\widehat{D}^g$ 
tells us that 
\begin{equation} \label{schr-lichn}
\int_M \left( 
|\widehat{\nabla}^g \psi_s|^2 
+ \frac{\scal^g + n(n-1)}{4} |\psi_s|^2 
\right) \, d\mu^g
= \frac{1}{4} H_{\Phi} (V_s),
\end{equation}
see \cite[(4.11) and (4.22)]{ChruscielHerzlich}. 

We denote by $H$ the space of positive smooth functions on 
$\bH^n \setminus B_{R_0}$ which satisfy the Yamabe equation 
\eqref{eqYamabeb} and tend to $1$ at infinity. In the proof of 
Theorem \ref{thmNearPMT2} we use the functionals $\cF_s$ defined 
for $U \in H$ by 
\[
\cF_s (U) \definedas 
\inf
\left\{
\int_{\bH^n \setminus B_{R_1}}
|\widehat{\nabla}^g \psi|_g^2 
\, d\mu^g
\mathrel{} \middle| \mathrel{}
\psi - \Phi^*\phi_s \to 0 \text{ at infinity}
\right\}
\]
where $g = U^\kappa b$ and $s \in \Sigma$.
The infimum is attained by a spinor satisfying
\begin{equation*} 
\left\{
\begin{aligned}
(\widehat{\nabla}^g)^* \widehat{\nabla}^g \psi 
&\definedas
- \left( \nabla^g_{e_i} - \frac{i}{2} e_i \cdot \right)
\left(\nabla^g_{e_i} + \frac{i}{2} e_i \cdot \right) \psi
= 0, \\
\widehat{\nabla}^g_{\nu} \psi &= 0 
\text{ at the inner boundary of $(\bH^n \setminus B_{R_1}, g)$}, \\
\psi - \Phi^*\phi_s &\to 0 \text{ at infinity.} 
\end{aligned}
\right.
\end{equation*}

The following Lemma is similar to \cite[Lemma~12, page~231]{BrayPenrose}.

\begin{lemma} \label{Fscontinuous}
$\cF_s$ is continuous with respect to the $C^1$ topology on $H$. 
\end{lemma}

\begin{proof}
Let $U_1$, $U_2$ be functions in $H$ and set $g_1 = U_1^{\kappa} b$, 
$g_2 = U_2^{\kappa} b = W^{\kappa} g_1$, where $W \definedas U_2/U_1$.
Let $\psi_1$ and $\psi_2$ be the minimizers for $\cF_s(U_1)$ and
$\cF_s(U_2)$. Using standard methods of identifying spinors for 
conformal metrics (see for example \cite[Section~5.2]{Hijazi}) we 
identify the spinor $\psi_1$ defined for the metric $g_1$ with the 
spinor $\overline{\psi_1}$ for the metric $g_2$. Further, we can 
express the covariant derivative $\nabla^{g_2} \overline{\psi_1}$
as a leading term which is $\nabla^{g_1} \psi_1$ followed by terms
involving $dW$ and $\psi_1$. We then compute
\[ \begin{split}
\cF_s(U_2)
&= \int_{\bH^n \setminus B_{R_1}}
|\widehat{\nabla}^{g_2} \psi_2|_{g_2}^2 
\, d\mu^{g_2} \\
&\leq \int_{\bH^n \setminus B_{R_1}}
|\widehat{\nabla}^{g_2} \overline{\psi_1}|_{g_2}^2 
\, d\mu^{g_2} \\
&= \int_{\bH^n \setminus B_{R_1}}
|\widehat{\nabla}^{g_1} \psi_1 |_{g_1}^2 
\, d\mu^{g_1}
+ E(W,\psi_1) \\
&=
\cF_s(U_1) + E(W,\psi_1).
\end{split} \]
Here the remainder $E(W,\psi_1)$ is given by an integral over 
$\bH^n \setminus B_{R_1}$ where each term in the integrand is quadratic 
in $\psi_1$ (containing $\psi_1$ or $\nabla^{g_1} \psi_1$) and contains 
one or two factors of the type $(1 - W^q)$ or $dW$. Since the minimizing 
spinor $\psi$ for $\cF_s(U)$ depends continuously on $U$ we conclude 
that $E(W,\psi_1)$ can be made arbitrarily small by choosing $U_2$ 
sufficiently close to $U_1$ in $C^1$. By interchanging $U_1$ and $U_2$
we get an inequality in the other direction, and we conclude that 
$\cF_s$ is continuous.
%
\end{proof}

Let $s^{\pm} \in \Sigma$ be such that 
\[
V_{s^{\pm}} = |\phi_{s^{\pm}}|^2 = V_{(0)} \pm V_{(1)}.
\] 
We define the functional $\cF$ by 
\[
\cF(U) \definedas \cF_{s^{+}}(U) + \cF_{s^{-}}(U) .
\]


In the next Lemma we prove that the mass bounds $\cF(U)$.

\begin{lemma} \label{massboundscF}
For $(M,g,\Phi,U) \in \cA^{\spin}$ we have
\[
\cF(U) \leq (n-1)\omega_{n-1} m^g.
\]
\end{lemma}

\begin{proof}
Since the integral in the definition of $\cF_s(U)$ lacks the non-negative 
term involving scalar curvature and is taken over a smaller domain it is 
never larger than the integral in \eqref{schr-lichn}. Further, the infimum 
in the definition of $\cF_s(U)$ can only decrease the value of the integral 
in \eqref{schr-lichn} and we conclude that 
\[
\cF_{s^{\pm}}(U) \leq
\frac{1}{4} H_{\Phi} (V_{s^{\pm}}).
\]
Therefore
\[ \begin{split}
\cF(U) &= \cF_{s^{+}}(U) + \cF_{s^{-}}(U) \\
&\leq
\frac{1}{4}\left( H_{\Phi} (V_{s^+}) + H_{\Phi} (V_{s^-}) \right) \\
&=
\frac{1}{4}\left( H_{\Phi} (V_{(0)} + V_{(1)}) 
+ H_{\Phi} (V_{(0)} - V_{(1)}) \right) \\
&=
\frac{1}{2} H_{\Phi} (V_{(0)}) \\
&= 
(n-1)\omega_{n-1} m^g
\end{split} \]
follows from \eqref{confmassbalanced}.
\end{proof}


\begin{lemma} \label{cF=0}
$\cF(U) = 0$ if and only if $U \equiv 1$.
\end{lemma}

\begin{proof}
If $\cF(U) = 0$ then $\cF_{s^{+}}(U)  = 0$ and $\cF_{s^{-}}(U) = 0$ and 
both the infima are attained by non-trivial Killing spinors. The 
existence of a non-trivial Killing spinor implies that $g$ is an 
Einstein metric with scalar curvature $-n(n-1)$. Since the metric $g$ 
is also conformally flat it must have constant negative curvature $-1$. 
Since $U \to 1$ at infinity we conclude that $U \equiv 1$. If 
$U \equiv 1$ then $g$ is the hyperbolic metric which has Killing spinors, 
and thus $\cF(U) = 0$.
\end{proof}

\begin{proof}[Proof of Theorem \ref{thmNearPMT2}]
As in the proof of Theorem \ref{thmNearPMT1} we argue by contradiction.
We assume that there is a sequence $(M_k, g_k, \Phi_k, U_k)$ of elements 
of $\cA^{\spin}$ such that $m^{g_k}$ tends to zero while 
$\left| U_k - 1 \right| \geq \epsilon$. Arguing as in the proof of 
Theorem \ref{thmNearPMT1}, a subsequence of $U_k - 1$ will converge to 
a limit $U_{\infty} - 1$ in $X^{2, p}_\delta (\bH^n \setminus B_{R_1})$ for 
which $\sup_{\bH^n \setminus B_{R_1}} \left| U_{\infty} - 1 \right| 
\geq \epsilon $. From Lemma \ref{Fscontinuous} we see that 
$\lim_{k \to \infty} \cF(U_k) = \cF(U_{\infty})$, and from Lemma 
\ref{massboundscF} we have $\lim_{k \to \infty} \cF(U_k) = 0$. 
Lemma \ref{cF=0} then tells us that $U_{\infty} \equiv 1$ which is a 
contradiction. From this we conclude that for every $\epsilon > 0$ 
there is a $\delta > 0$ such that for $(M, g, \Phi, U)$ belonging to 
$\cA^{\spin}$ with $m^g \leq \delta$ it holds that
$\sup_{\bH^n \setminus B_{R_1}} \left|U-1\right| \leq \epsilon$.
The theorem now follows from Proposition \ref{propUControlledDecay}.
\end{proof}

\appendix
\section{The anti-de~Sitter-Schwarzschild spacetime} 
\label{secAdS}


In this appendix, we discuss the anti-de~Sitter-Schwarzschild metrics in 
dimension $n$ following \cite[Section 2]{ShiTam} where only the case
$n=3$ is treated. These metrics are also called Kottler metrics with
negative cosmological constant. Furthermore we explicit the lapse
function, see \cite{StuchlikHledik} for the $3+1$-dimensional case.

\subsection{The metric in areal coordinate}

Let $g$ be a Riemannian metric on the $n$-dimensional manifold $M$ and
let 
\[
\gamma \definedas -V^2 dt^2 + g
\]
be a Lorentzian metric defined on the manifold $\cM \definedas \bR \times M$. 
If we assume that the function $V$ does not depend on $t$ then $\gamma$ 
solves the Einstein equations with cosmological constant $\Lambda$,
\[
\ric^{\gamma} - \frac{\scal^{\gamma}}{2} \gamma + \Lambda \gamma = 0,
\]
if and only if
\begin{subequations}
\begin{align}
\scal^g &= 2 \Lambda, \label{eqStatic1}\\
\ric^g - \frac{\hess^g V}{V} + \frac{\Delta^g V}{V} g &= 0.
\label{eqStatic2}
\end{align}
\end{subequations}

Such a metric $g$ is called static. See also
\cite[Equations (0.1)-(0.2)]{ChruscielHerzlich}. In what follows we will
always assume that when the metric is not indicated, the curvature tensors
and the connection are defined with respect to the metric $g$. We are
interested in the case of negative cosmological constant. We assume that 
\[
\Lambda = -\frac{n(n-1)}{2},
\]
which can always be achieved by a rescaling. From Equation 
\eqref{eqStatic1}, this imposes $\scal^g = -n(n-1)$. Taking the trace of 
Equation \eqref{eqStatic2} yields $\Delta V = n V$ so we can
rewrite the system as
\begin{subequations}
\begin{align}
\scal^{g} & = -n(n-1), \label{eqStatic3}\\
\ric^g_{ij} - \frac{\hess^g_{ij} V}{V} + n g_{ij} & = 0.
\label{eqStatic4}
\end{align}
\end{subequations}

Note that the slice $t=0$ is totally geodesic. In particular, marginally
(outer) trapped surfaces correspond to minimal surfaces for the metric 
$g$. We now assume that the metric $g$ is rotationally symmetric. 
Such a metric can be written in full generality as 
$g = ds^2 + k(s)^2 \sigma$ where $\sigma$ is the round metric on $S^{n-1}$. 
The mean curvature of a surface of constant $s$ is given by 
$H(s) = (n-1)\frac{\partial_s k}{k}$. So, in a region where no surface
of constant $s$ is a minimal surface, $k$ has non-vanishing derivative 
and we can use it as a radial (areal) coordinate $\rho$ so that
\[
g = f(\rho)^2 d\rho^2 + \rho^2 \sigma.
\]
In what follows we assume that the coordinate index $1$ corresponds to 
the $\rho$-coordinate while upper-case latin letters represent coordinates 
on the sphere and run from $2$ to $n$. The Christoffel symbols of the 
metric $g$ are given by
\[ \left\lbrace
\begin{aligned}
\Gamma^{1}_{11} & = \frac{f'}{f},\\
\Gamma^{1}_{1A} & = 0,\\
\Gamma^{A}_{11} & = 0,\\
\Gamma^{A}_{1B} & = \frac{1}{\rho} \delta^A_B,\\
\Gamma^{1}_{AB} & = -\frac{\rho}{f^2} \sigma_{AB},\\
\Gamma^{C}_{AB} & = \gamma^C_{AB},
\end{aligned}
\right.\]
where $\gamma^C_{AB}$ are the Christoffel symbols of the metric $\sigma$. 
The components of the curvature tensors of the metric $g$ can then 
be computed,
\[\left\lbrace
\begin{aligned}
\riemuddd{L}{K}{I}{J} & = \left(1 - \frac{1}{f^2}\right)
  \left(\delta^L_I \sigma_{JK} - \delta^L_J \sigma_{IK}\right),\\
\riemuddd{1}{K}{I}{J} & = 0,\\
\riemuddd{1}{K}{1}{J} & = \frac{\rho f'}{f^3} \sigma_{JK},\\
\ricdd{K}{J} & = \left[(n-2) \left(1 - \frac{1}{f^2}\right)
  + \frac{\rho f'}{f^3}\right] \sigma_{JK},\\
\ricdd{1}{J} & = 0,\\
\ricdd{1}{1} & = (n-1) \frac{\rho f'}{f},\\
\scal & = 2(n-1) \frac{f'}{\rho f^3}
  + \frac{(n-1)(n-2)}{\rho^2} \left(1 - \frac{1}{f^2}\right).
\end{aligned}
\right.\]

What is interesting about these formulas is that the curvature 
tensor depends only on the first derivative of $f$. In particular, 
Equation \eqref{eqStatic4} becomes
\begin{equation}\label{eqConstScalCurv}
-n = 2 \frac{f'}{\rho f^3} 
+ \frac{n-2}{\rho^2} \left(1 - \frac{1}{f^2}\right). 
\end{equation}
Defining $u$ by the relation $f(\rho) = u(\rho)^{-\frac{1}{2}}$ we get
\[
\frac{u'}{\rho} + (n-2) \frac{u-1}{\rho^2} = n.
\]
The general solution of this equation is given by 
\[
u(\rho) = 1 + \rho^2 - \frac{2m}{\rho^{n-2}}
\] 
where $m$ is a free parameter that can be identified with the mass. 
Our next goal is to find the lapse function $V = V(\rho)$. Equation 
\eqref{eqStatic4} can be decomposed into radial, tangential, and mixed 
components. The mixed components vanish while the other two 
lead to the following equations.
\begin{equation}\label{eqLapse}
\left\lbrace
\begin{aligned}
0 &= (n-1) V  \frac{\partial_\rho f}{\rho f} - \partial_\rho^2 V
  + \frac{\partial_\rho f}{f} \partial_\rho V + n f^2 V,\\
0 &= V \left[(n-2)\left(1- \frac{1}{f^2}\right)
  + \frac{\rho \partial_\rho f}{f^3}\right]
  - \frac{\rho}{f^2} \partial_\rho V + n \rho^2 V.
\end{aligned}
\right.
\end{equation}
The second equation can be combined with Equation \eqref{eqConstScalCurv} 
to yield
\[
\frac{V'}{V} = - \frac{f'}{f}.
\] 
Hence $V = \frac{\lambda}{f}$. Up to a redefinition of $t$ we can assume 
that $V = \frac{1}{f}$. It is then checked by a simple calculation that the 
first line of Equation \eqref{eqLapse} is also fulfilled. Hence the 
anti-de~Sitter-Schwarzschild metric can be written as follows, 
\[
\gamma_{\adss} = 
- \left(1+\rho^2 - \frac{2m}{\rho^{n-2}}\right) dt^2
+ \frac{d\rho^2}{1+\rho^2 - \frac{2m}{\rho^{n-2}}} + \rho^2 \sigma.
\]

We now study separately the cases $m > 0$ and $m < 0$.

\subsection{The case of positive mass}

The metric 
\[g_{\adss} =
\frac{d\rho^2}{1+\rho^2 - \frac{2m}{\rho^{n-2}}} + \rho^2 \sigma\]
is only defined on the set $\{\rho \geq a(m)\}$ where 
$a(m)$ is the unique solution of the equation
\[
1 + \rho^2 - \frac{2m}{\rho^{n-2}} = 0.
\]

We define
\[ 
h_m(\rho) 
= \int_\rho^\infty \frac{ds}{s \sqrt{1+s^2-\frac{2m}{s^{n-2}}}}
\]
and the functions $r$ and $\phi$ by
\[
e^r = \frac{1+e^{-h_m(\rho)}}{1-e^{-h_m(\rho)}},
\qquad 
\phi^{\frac{2}{n-2}} = \frac{\rho}{\sinh r}.
\]
We note that $r \to r_0(m) > 0$ when $\rho \to a(m)^+$. The function
$r : (a(m), \infty) \to (r_0(m), \infty)$ is a smooth increasing function 
of $\rho$. We remark that
\begin{equation}\label{eqRelationRhoR}
\left\lbrace
\begin{aligned}
\rho 
&= 
\phi^{\frac{2}{n-2}} \sinh r , \\
\frac{d\rho}{\rho\sqrt{1+\rho^2 - \frac{2m}{\rho^{n-2}}}} 
&= 
\frac{dr}{\sinh r}.
\end{aligned}
\right.
\end{equation}
The metric $g_{\adss} $ can then be written as
\[
g_{\adss}  = \phi^{\frac{4}{n-2}} \left( dr^2 + \sinh^2 r \sigma \right).
\]

The mean curvature of the hypersurfaces of constant $\rho$ is given by
\begin{equation}\label{eqMeanCurvConstRho}
H = \phi^{-\frac{2}{n-2}} \left[(n-1) \coth r 
+ \frac{2(n-1)}{n-2} \frac{\partial_r \phi}{\phi}\right].
\end{equation}
A simple calculation shows that $\phi$ and $\partial_r \phi$ are 
continuous at $r = r_0(m)$ and that the hypersurface $r = r_0(m)$ is a 
minimal surface.

We next show that the manifold can be doubled to a complete 
asymptotically hyperbolic manifold of constant scalar curvature. 
For this we first switch to the conformal ball model of hyperbolic 
space and set
\[
\tau = \frac{e^r - 1}{e^r + 1}.
\] 
The metric $g_{\adss} $ becomes
\[
g_{\adss}  = \frac{4 \phi^{\frac{4}{n-2}}}{(1-\tau^2)^2} 
\left(d\tau^2 + \tau^2 \sigma \right)
\]
and is defined on the annulus $b(m) \leq |x| < 1$ in $\bR^n$ where 
$\tau=|x|$ and $b(m)$ is given by
\[
b(m) = \frac{e^{r_0(m)} - 1}{e^{r_0(m)} + 1}.
\] 
As is well known, the inversion on $\bR^n \setminus \{0\}$ given by
\[
i : x \mapsto \frac{b(m)^2}{|x|^2} x
\]
is a conformal transformation. Pulling back the metric $g$ to the
annulus $b^2(m) < |x| \leq b(m)$ by $i$, we get the following
extension of the metric $g_{\adss} $,
\[
g _{\adss} = 
\frac{4 f_m^{\frac{4}{n-2}}}{(1-\tau^2)^2} 
\left( d\tau^2 + \tau^2 \sigma \right),
\]
where
\[
f_m^{\frac{2}{n-2}}(\tau) =
\begin{dcases}
\phi^{\frac{2}{n-2}}(\tau) 
&\text{if $b(m) \leq \tau \leq 1$,} \\
b(m)^2 \phi^{\frac{2}{n-2}}\left(\frac{b(m)^2}{\tau}\right)
\frac{1-\tau^2}{\tau^2 - b(m)^4}
&\text{if $b(m)^2 \leq \tau \leq b(m)$.}
\end{dcases}
\]

In the following propositions we collect the basic properties of the 
metric $g_{\adss}$. Most of them are useful in the course of the 
proof of our main results. See also \cite[Section 2]{ShiTam} for 
the three-dimensional case.

\begin{proposition} \label{propAdSSchPos1}
For each $m > 0$ the anti-de~Sitter-Schwarzschild metric is 
asymptotically hyperbolic and is defined on $\bH^n \setminus B_{b(m)^2}$. 
Moreover,
\begin{enumerate}
\item 
$f_m > 1$, $\lim_{\tau \to 1} f_m(\tau) = 1$,
$\lim_{\tau \to b(m)^2} f_m(\tau) = \infty$. 
\item 
There exists a constant $C>0$ independent of $m$ such that 
$f_m \leq 1+Cme^{-nr}$ provided that $r\geq r_1(m)$, where $r_1$ is a 
non-decreasing continuous function of $m$ such that $r_1(m)>r_0(m)$.
Consequently, $f_m = 1 + O(e^{-n r})$ when
$r \to \infty$. 
\item
$g_{\adss}$ has constant scalar curvature $-n(n-1)$ and mass $m$.
\item 
$\partial_r f_m < 0$.
\item 
The hypersurface $r = r_0(m)$ is the only compact minimal surface.
\end{enumerate}
\end{proposition}

\begin{proof} 
The mass of $g_{\adss}$ is easily computed using 
\cite[Formula (2.25)]{ChruscielHerzlich}, and Property 3 follows. 

Fixing $r \geq r_0(m)$, we remark that for all $s \in (a(m), \infty)$,
\[
h_m(s) > h_0(s) = \frac{1}{2} \ln \frac{\sqrt{1+s^2}+1}{\sqrt{1+s^2}-1}.
\]
In particular, if $\sinh r \geq a(m)$ we get
\[
h_m(\sinh r) > \ln \left(\coth \frac{r}{2}\right).
\]
Since $h_m(\rho)=\ln \coth \frac{r}{2}$, and since $h_m$ is decreasing, 
we have $\rho > \sinh r$. It is also obvious that $\sinh r < \rho$ if 
$\sinh r < a(m)$. This proves that $\phi = 
\left(\frac{\rho}{\sinh r} \right)^{\frac{n-2}{2}}> 1$ 
for any $r \geq r_0(m)$. 

We next find an upper bound for $f_m$. First, it is clear that 
$a(m) < (2m)^{1/n}$. If we assume that  $\rho \geq (2 m)^{1/n}$ then 
$\rho > a(m)$ and
\[\begin{split}
h_m(\rho) 
&= 
\int_\rho^\infty \frac{ds}{s\sqrt{1+s^2}\sqrt{1-\frac{2m}{(1+s^2)s^{n-2}}}}\\
&\leq 
\frac{1}{\sqrt{1-m \rho^{-n}}} \int_\rho^\infty 
\frac{ds}{s\sqrt{1+s^2}}\\
&= 
\frac {\arcsinh (\rho^{-1})}{\sqrt{1-m \rho^{-n}}}.
\end{split}\]
Observe also that $\sinh r=\frac{1}{\sinh h_m(\rho)}$, hence
$\frac{\rho}{\sinh r}=\rho \sinh h_m(\rho)$. Set 
\[
\eta(t) \definedas 
\sinh \left(\frac{\arcsinh (\rho^{-1})}{\sqrt{1-t}}\right).
\]
Let $R_1>0$ be fixed, and assume that 
$\rho\geq \max\{r^{-1}(R_1), (2 m)^{1/n}\}$. Using the mean value theorem 
and the inequality $\arcsinh (\rho^{-1})\leq \rho^{-1}$ we have
\[\begin{split}
\sinh h_m(\rho)
&\leq
\eta(m\rho^{-n}) \\
&\leq 
\eta(0) + m \rho^{-n} \sup_{0\leq \theta \leq 1} \eta'(\theta m \rho^{-n}) \\
&=
\rho^{-1} + m \rho^{-n} \sup_{0\leq \theta \leq 1}
\left(
\frac{\arcsinh (\rho^{-1})}{2(1-\theta m \rho^{-n})^{3/2}}
\cosh \left(\frac{\arcsinh (\rho^{-1})}{\sqrt{1-\theta m \rho^{-n}}}\right)
\right)\\
&\leq \rho^{-1}\left(1+C m \rho^{-n}\right),
\end{split}\]
for some constant $C>0$ which does not depend on $m$. Since 
$\rho \geq \sinh r$ it is now easy to check that
$f_m=(\rho \sinh h_m(\rho))^{\frac{n-2}{2}}\leq 1+C m e^{-nr}$ 
(possibly for a larger constant $C>0$), provided that
$r \geq r_1(m) \definedas \max\{R_1, r((2 m)^{1/n})\}$. By definition, it 
is clear that $r_1(m)>r_0(m)$. The second statement is thereby proved. 

Next, $f_m$ solves the Yamabe equation
\[
-\frac{4(n-1)}{n-2} \Delta^b f_m 
+ n(n-1) \left(f_m^{\frac{n+2}{n-2}}-f_m\right)=0.
\]
In polar normal coordinates we have $\sqrt{\det(b)} = \sinh^{n-1} r$. 
Hence, from the well known formula 
\[
\Delta^b f_m 
= \frac{1}{\sqrt{\det(b)}} \partial_i 
\left(\sqrt{\det(b)} b^{ij} \partial_j f_m\right)
\]
we infer that
\[
\partial_r \left(\sinh^{n-1} r  \partial_r f_m\right) 
= \frac{n(n-2)}{4} \sinh^{n-1} r \left(f_m^{\frac{n+2}{n-2}}-f_m\right).
\]
Assume that $\partial_r f_m(\bar{r}) \geq 0$ for some $\bar{r}$. Then, 
$\partial_r f_m(r) > 0$ for all $r > \bar{r}$ since $f_m > 1$.
This contradicts the fact that $f_m \to 1$ when $r \to \infty$, 
$f_m > 1$. Hence, $\partial_r f_m < 0$ for all $r$.

We finally prove that the hypersurface $r = r_0(m)$ is the only minimal
surface. From Formula \eqref{eqMeanCurvConstRho}, the sphere of constant 
$\rho$ has mean curvature
\[
H(\rho) = (n-1) \sqrt{1+\frac{1}{\rho^2}-\frac{2m}{\rho^n}}.
\]
For any $\rho > a(m)$ we have $H(\rho) > 0$. Thus by the maximum principle 
for minimal surfaces, if $\Sigma$ is a minimal surface, then 
$\sup_{\Sigma} \rho \leq a(m)$, that is $\sup_{\Sigma} \tau \leq b(m)$. By 
symmetry, we also have that $\inf_{\Sigma} \tau \leq b(m)$. This proves 
that $\Sigma$ coincides with the sphere $r = r_0(m)$.
\end{proof}

\begin{proposition} \label{propAdSSchPos2}
$a(m)$, $r_0(m)$ and $b(m)$ are continuous increasing functions of $m$. 
Further,
\begin{enumerate}
\item 
$a(m), r_0(m), b(m) \to 0$ as $m \to 0$,
\item 
$a(m), r_0(m) \to \infty$ and $b(m) \to 1$ as $m \to \infty$.
\end{enumerate}
\end{proposition}

\begin{proof}
It is easy to see that $a(m)$ is a continuous increasing function of 
$m$. Since the function $\rho \mapsto 1 + \rho^2 - \frac{2m}{\rho^{n-2}}$ 
is increasing, we know that $\rho_-(m) \leq a(m) \leq \rho_+(m)$ 
provided that
\[\left\lbrace\begin{aligned}
 0 & \leq 1 + \rho_+^2 - \frac{2m}{\rho_+^{n-2}},\\
 0 & \geq 1 + \rho_-^2 - \frac{2m}{\rho_-^{n-2}}.
\end{aligned}\right.\]
One can select $\rho_+ = (2m)^{1/n}$. Assuming $m > 1$, we choose
\[
\rho_- = (2m)^{1/n} \sqrt{1 - \frac{1}{(2m)^{2/n}}}.
\]
Simple computations show that both inequalities are fulfilled. Hence 
for large $m$, $a(m) \sim (2m)^{1/n}$. For small positive $m$, we
obviously have $0 < a(m) < \rho_+(m)$. So $a(m) \to 0$ when $m \to 0^+$.

We next turn our attention to the function $r_0$. We first give an
upper bound for $h_m(a(m))$ as follows. Note that on the interval
$(a(m), \infty)$ we have
\[\begin{split}
1 + s^2 - \frac{2m}{s^{n-2}} 
&\geq 
1 + s^2 - \frac{2m}{a(m)^{n-2}} \\
&= 
1 + s^2 - (1+a(m)^2) \\
&= 
s^2 - a(m)^2.
\end{split}\]
Hence,
\[
\ln\left(\coth \frac{r_0(m)}{2}\right)
= h_m(a(m))
\leq \int_{a(m)}^\infty \frac{ds}{s \sqrt{s^2-a(m)^2}}
= \frac{\pi}{2a(m)}.
\]
This implies that $r_0(m) \to \infty$ as $m \to \infty$.

In order to estimate $r_0$ when $m \to 0^+$ we give a lower bound for
$h_m(a(m))$, assuming $a(m)<1$,
\[\begin{split}
\ln \left(\coth \frac{r_0(m)}{2}\right)
&= 
h_m(a(m)) \\
&= 
\int_{a(m)}^\infty 
\frac{ds}{s \sqrt{1+s^2 - \frac{a(m)^n + a(m)^{n-2}}{s^{n-2}}}}\\
&= 
\int_{a(m)}^\infty 
\frac{ds}{s^{2-\frac{n}{2}} \sqrt{s^n - a(m)^n + s^{n-2} - a(m)^{n-2}}}\\
&\geq
\int_{a(m)}^\infty 
\frac{ds}{s^{2-\frac{n}{2}} \sqrt{(s-a(m))(n s^{n-1} + (n-2) s^{n-3})}}\\
&\geq 
\int_{a(m)}^\infty \frac{ds}{\sqrt{s} \sqrt{(s-a(m))(n s^2 + (n-2))}}\\
&\geq 
\frac{1}{\sqrt{2n-2}} \int_{a(m)}^1 \frac{ds}{\sqrt{s(s-a(m))}}\\
&=
\frac{1}{\sqrt{2n-2}} \int_1^{\frac{1}{a(m)}} \frac{dt}{\sqrt{t(t-1)}}.
\end{split}\]
It is obvious that the last integral diverges when $a(m) \to 0^+$. Hence 
$r_0(m) \to 0$ when $m \to 0$.

The limits of $b$ follow from the relation 
$b(m) = \frac{e^{r_0(m)}-1}{e^{r_0(m)}+1}$.
\end{proof}

\subsection{The case of negative mass}

Remark that when $m < 0$ the function $h(m)$ tends to a finite
positive value at $\rho=0$. Changing to the $r$ coordinate,
this means that the metric
$g = \phi^{\frac{4}{n-2}} (dr^2 + \sinh^2 r \sigma)$ is
only defined for $r \geq r_0(m)$ such that
\[
h_m(0) 
= \int_0^\infty \frac{ds}{s \sqrt{1+s^2 - \frac{2m}{s^{n-2}}}}
= \ln \frac{1+e^{-r_0(m)}}{1-e^{-r_0(m)}}.
\]
The function $\phi$ satisfies $\phi(r_0(m)) = 0$.

\begin{proposition} \label{propAdSSchNeg}
The function $m \mapsto r_0(m)$ is continuous and strictly
decreasing on the interval $(-\infty, 0)$. Further,
\begin{enumerate}
\item
\[\left\lbrace
\begin{aligned}
\lim_{m \to 0^-} r_0(m) &= 0 \\
\lim_{m \to -\infty} r_0(m) &= \infty \\
\end{aligned}
\right.\]
\item
The function $f_m\definedas\phi : \bH^n \setminus B_{r_0(m)}(0) \to \bR_+$ 
solves the Yamabe equation with zero boundary value on
$\partial B_{r_0(m)}(0)$ and satisfies $f_m<1$. 
\item
There exists a constant $C>0$ independent of $m$ such that 
$f_m \geq 1-Cme^{-nr}$ provided that $r\geq r_1(m)$, where $r_1$ is a 
non-increasing continuous function of $m$ such that $r_1(m)>r_0(m)$. 
Consequently, $f_m = 1 + O(e^{-n r})$ when $r \to \infty$. 
\item 
$\partial _r f_m>0$.
\end{enumerate}
\end{proposition}

\begin{proof}
We remark that the integrand is positive and strictly increasing with
respect to $m$. From dominated convergence, it is easy to argue that
$m \mapsto h_m(0)$ is continuous. When $|m| \to \infty$, the integrand
tends to $0$ so $\lim_{m \to - \infty} h_m(0) = 0$. This forces
$\lim_{m \to - \infty} r_0(m) = \infty$. Similarly, when $m \to 0^-$,
by the monotone convergence theorem,
\[
h_m(0) \to \int_0^\infty \frac{ds}{s \sqrt{1+s^2}} = \infty.
\]
Hence, $\lim_{m \to 0^-} r_0(m) = 0$. 

The properties of $f_m$ follow in the same manner as their counterparts
in the case $m>0$ (see Proposition \ref{propAdSSchPos1}). We only remark
 that having fixed $R_1>0$ one may define 
$r_1(m)$ as $r_1(m)\definedas \max\{R_1, r((-Cm)^{1/n})\}$, where the 
constant $C>0$ depends on $R_1$ only. It is then obvious that 
$r_1(m)>r_0(m)$.
\end{proof}

\subsection{A characterization of anti-de~Sitter-Schwarzschild 
spacetimes}

In this section, we give a characterization of anti-de~Sitter-Schwarzschild 
metrics which is useful in the proof of Theorem \ref{thmNearPMT1}. See 
\cite{KobayashiObata} for similar results.

\begin{proposition} \label{propLCFStatic}
Let $K$ be a compact subset of $\bH^n$ such that $\bH^n \setminus K$ is
connected and let $U, V$ be two functions defined on $\bH^n \setminus K$. 
Let $g \definedas U^\kappa b$. Assume that the metric 
\[
-V^2 dt^2 + g
\] 
is static with cosmological constant
\[
\Lambda = -\frac{n(n-1)}{2}.
\]
Assume further that the function $U$ is bounded from above and away from 
zero and that the function $V$ is positive, tends to infinity at infinity
and has no critical point outside a compact set. Then there is a point 
$x_0 \in \bH^n$ and $m \in \bR$ such that 
\[
U = f_m(r),
\]
where $r \definedas d^b(x_0, \cdot)$.
\end{proposition}

Before diving into the proof, we explain briefly the underlying idea.
The main aim is to prove that the metric $g$ and the lapse function $V$ 
are spherically symmetric around a point in $\bH^n$. A first indication
of this fact is Equation \eqref{eqCotton4} which proves that the Ricci
tensor has at most two distinct eigenvalues, one with multiplicity $1$
in the direction of the gradient of $V$ and another one with multiplicity
$n-1$ on the orthogonal hyperplane. Another indication is given by
Formula \eqref{eqSFF} which proves that the level sets of $V$ are umbilic
with constant sectional curvature. These two indications prove that
the metric is actually a warped product (Formula \eqref{eqWarped})
and some further estimates on $U$ allow us to conclude that $U$ 
coincides with $f_m$ for some $m$.

\begin{proof}
In what follows, covariant derivatives and curvatures are defined with 
respect to the metric $g$ unless stated otherwise.

Since $g = U^\kappa b$ on $\bH^n \setminus B_{R_1}$ is conformally flat,
it has vanishing Cotton-York tensor (see for example
\cite[Proposition 1.62]{HamiltonRicci}). Since $g$ has constant scalar
curvature this is equivalent to
\begin{equation}\label{eqCotton}
\nabla_i \tlric_{jk} - \nabla_j \tlric_{ik} = 0. 
\end{equation}
From the static equations \eqref{eqStatic3}-\eqref{eqStatic4} it 
follows that
\[
\tlric = \frac{\tlhess~V}{V},
\] 
where $\tlric \definedas \ric + (n-1)g$ denotes the traceless Ricci tensor
and $\tlhess~V$ denotes the traceless Hessian of $V$ (which in index 
notation is denoted by $\tlhessdd{i}{j} V$). From \eqref{eqCotton} 
and the fact that
\[
\frac{\hess~V}{V} 
= \frac{1}{V} \left( \tlhess~V + \frac{\Delta V}{n} g \right) 
= \frac{\tlhess~V}{V} + g,
\]
we conclude
\[\begin{split}
0 
&= 
\nabla_i \frac{\tlhessdd{j}{k} V}{V} 
- \nabla_j \frac{\tlhessdd{i}{k} V}{V}\\
&= 
\nabla_i \frac{\hessdd{j}{k} V}{V} 
- \nabla_j \frac{\hessdd{i}{k} V}{V}\\
&= 
\frac{\nabla_i \nabla_j \nabla_k V}{V} 
- \frac{\nabla_j \nabla_i \nabla_k V}{V} 
- \frac{\hessdd{j}{k} V}{V} \frac{\nabla_i V}{V} 
+ \frac{\hessdd{i}{k} V}{V} \frac{\nabla_j V}{V}\\
&= 
- \riemuddd{l}{k}{i}{j} \frac{\nabla_l V}{V} 
- \frac{\hessdd{j}{k} V}{V} \frac{\nabla_i V}{V} 
+ \frac{\hessdd{i}{k} V}{V} \frac{\nabla_j V}{V}.
\end{split}\]
Since $g$ is conformally flat its Weyl tensor vanishes, so
\[
\riem 
= \frac{\scal}{2n(n-1)} g \kulk g + \frac{1}{n-2} \tlric \kulk g 
= -\frac{1}{2} g \kulk g + \frac{1}{n-2} \tlric \kulk g,
\]
where $\kulk$ denotes the Kulkarni-Nomizu product (see for example
\cite[Definition 1.110]{Besse}). As a consequence, we get
\begin{equation} \label{eqCotton2}
\begin{split}
0 &= 
- \riemuddd{l}{k}{i}{j} \frac{\nabla_l V}{V} 
- \frac{\hessdd{j}{k} V}{V} \frac{\nabla_i V}{V} 
+ \frac{\hessdd{i}{k} V}{V} \frac{\nabla_j V}{V} \\
&= 
\left(g_{li} g_{kj} - g_{lj} g_{ki}\right) \frac{\nabla^l V}{V} 
- \frac{\hessdd{j}{k} V}{V} \frac{\nabla_i V}{V} 
+ \frac{\hessdd{i}{k} V}{V} \frac{\nabla_j V}{V} \\
&\qquad 
- \frac{1}{n-2} \left(\tlric_{li} g_{kj} + \tlric_{kj} g_{li} 
- \tlric_{lj} g_{ki} - \tlric_{ki} g_{lj}\right) \frac{\nabla^l V}{V} \\
&= 
- \frac{\tlhessdd{j}{k} V}{V} \frac{\nabla_i V}{V} 
+ \frac{\tlhessdd{i}{k} V}{V} \frac{\nabla_j V}{V} \\
&\qquad 
- \frac{1}{n-2} \left(\frac{\tlhessdd{l}{i} V}{V} g_{kj} 
+ \frac{\tlhessdd{k}{j} V}{V} g_{li} 
- \frac{\tlhessdd{l}{j} V}{V} g_{ki} 
- \frac{\tlhessdd{k}{i} V}{V} g_{lj}\right) \frac{\nabla^l V}{V} \\
&= 
\frac{n-1}{n-2} \left(\frac{\tlhessdd{i}{k} V}{V} \frac{\nabla_j V}{V} 
- \frac{\tlhessdd{j}{k} V}{V} \frac{\nabla_i V}{V}\right) \\
&\qquad 
- \frac{1}{n-2} \left(\frac{\tlhessdd{l}{i} V}{V} 
\frac{\nabla^l V}{V} g_{kj} 
- \frac{\tlhessdd{l}{j} V}{V} \frac{\nabla^l V}{V} g_{ki} \right).
\end{split}
\end{equation}
We set $\xi_i \definedas 
\frac{\tlhessdd{i}{j} V}{V} \frac{\nabla^j V}{V}$. Contracting the 
previous equation with $\frac{\nabla^k V}{V}$ we get
\[
0 = \xi_i \frac{\nabla_j V}{V} - \xi_j \frac{\nabla_i V}{V}.
\]
This is possible only if $\xi$ and $\frac{\nabla V}{V}$ are colinear. 
We let $\lambda$ be the function such that 
$\xi = (n-1) \lambda \frac{\nabla V}{V}$. Equation \eqref{eqCotton2} 
then implies
\begin{equation*} 
0 = 
\left(\frac{\tlhessdd{i}{k} V}{V} 
+ \lambda g_{ik}\right)\frac{\nabla_j V}{V} 
- \left(\frac{\tlhessdd{j}{k} V}{V} 
+ \lambda g_{jk}\right)\frac{\nabla_i V}{V}.
\end{equation*}
This is possible only if 
\[
\frac{\tlhessdd{i}{k} V}{V} + \lambda g_{ik} 
= 
\mu \frac{\nabla_i V}{V} \frac{\nabla_k V}{V}
\] 
for some function $\mu$. The trace of this last equation gives a 
direct relation between $\lambda$ and $\mu$, 
\[
\lambda = \frac{\mu}{n} \left|\frac{dV}{V}\right|^2.
\] 
Hence,
\begin{equation}\label{eqCotton4}
\tlric_{ij} 
= \frac{\tlhessdd{i}{j} V}{V} 
= \mu \frac{\nabla_i V}{V} \frac{\nabla_j V}{V} 
- \frac{\mu}{n} \left|\frac{dV}{V}\right|^2 g_{ij}.
\end{equation}
By a straightforward calculation, we have
\begin{equation}\label{eqdV}
\nabla_i \left|\frac{dV}{V}\right|^2 
= 2 \left(1 + \left[\mu \left(1-\frac{1}{n}\right)-1\right] 
\left|\frac{dV}{V}\right|^2\right) \frac{\nabla_i V}{V}.
\end{equation}
We now choose $V_0$ to be such that $V$ has no critical point outside
$V^{-1}(-\infty, V_0)$. We remark that $V^{-1}(V_0, \infty)$ is connected.
Indeed if it was not, from the assumption that $V$ is proper it would
have one bounded connected component $\Omega$. Since $V = V_0$ on
$\partial \Omega$, $V$ reaches a local maximum on $\Omega$ which 
contradicts the assumption that $V$ has no critical point on $\Omega$.
We let $\Sigma_0$ be the boundary of a connected component of 
$V^{-1}(-\infty, V_0)$. Equation \eqref{eqdV} shows that 
$\left|\frac{dV}{V}\right|^2$ is constant along $\Sigma_0$. Plugging
Equations \eqref{eqCotton4} and \eqref{eqdV} into
\[
\nabla_i \frac{\tlhessdd{j}{k} V}{V} - \nabla_j \frac{\tlhessdd{i}{k} V}{V} 
= 0,
\]
we get
\begin{equation*} 
\begin{split}
0 
&= 
\nabla_i \mu \left( \frac{\nabla_j V}{V}\frac{\nabla_k V}{V} 
- \frac{1}{n} \left|\frac{dV}{V}\right|^2 g_{jk} \right)
- \nabla_j \mu \left(\frac{\nabla_i V}{V}\frac{\nabla_k V}{V} 
- \frac{1}{n} \left|\frac{dV}{V}\right|^2 g_{ik}\right) \\
&\qquad 
+ \mu \underbrace{\left[1 + \frac{2}{n} 
\left(1 - \left|\frac{dV}{V}\right|^2\right)
+ \frac{\mu}{n}\left(1-\frac{2}{n}\right) 
\left|\frac{dV}{V}\right|^2\right]}_{\displaystyle{=: \theta}}
\left(g_{ik} \frac{\nabla_j V}{V} - g_{jk} \frac{\nabla_i V}{V}\right).
\end{split}
\end{equation*}
Taking the trace of this last equation with respect to $j$ and $k$ we get
\[
\frac{1}{n} \left|\frac{dV}{V}\right|^2 \nabla_i \mu 
= 
\left[\left\<d\mu, \frac{dV}{V}\right\> 
+ (n-1) \theta \mu \right] \frac{\nabla_i V}{V}.
\]
This implies that $\mu$ is constant on the hypersurface $\Sigma_0$. The 
second fundamental form $S$ of $\Sigma_0$ is equal to the normalized 
Hessian of $V$ restricted to $T \Sigma_0$, that is
\begin{equation}\label{eqSFF}
S_{ij} 
= \frac{\hessdd{i}{j} V}{|dV|} 
= \frac{V}{|dV|} 
\left(1 - \frac{\mu}{n} \left|\frac{dV}{V}\right|^2\right) g_{ij}.
\end{equation}
Hence the hypersurface $\Sigma_0$ is umbilic with constant mean 
curvature. From the conformal transformation law of the second 
fundamental form, $\Sigma_0$ is umbilic for the hyperbolic metric
$b$ as well. Since $\Sigma_0$ is also compact it is a round sphere.

Note that the curvature of $\Sigma_0$ is given by the Gauss Formula,
\[
\riem^{\Sigma_0} = \riem + \frac{1}{2} S \kulk S.
\]
From the form of the Riemann tensor of $g$ and the special form of 
$S$, we immediately conclude that the metric induced on $\Sigma_0$ has 
constant curvature,
\[ \begin{aligned}
\riem^{\Sigma_0}
&= 
-\frac{1}{2} g \kulk g + \frac{1}{n-2} \tlric \kulk g 
+ \frac{1}{2} \frac{\hess V}{|dV|} \kulk \frac{\hess V}{|dV|} \\
&= 
-\frac{1}{2} g \kulk g 
+ \frac{1}{n-2} \frac{\tlhess V}{V} \kulk g 
+ \frac{1}{2} \frac{V^2}{\left|dV\right|^2}
\left(\frac{\tlhess V}{V} + g\right) \kulk 
\left(\frac{\tlhess V}{V} + g\right) \\
&= 
-\frac{1}{2} g \kulk g 
- \frac{1}{n-2} \frac{\mu}{n} \left|\frac{dV}{V}\right|^2 g \kulk g 
+ \frac{1}{2} \frac{V^2}{\left|dV\right|^2}
\left(1-\frac{\mu}{n} \left|\frac{dV}{V}\right|^2\right)^2 g \kulk g\\
&= 
\left[-\frac{1}{2} 
- \frac{1}{n-2} \frac{\mu}{n} \left|\frac{dV}{V}\right|^2 
+ \frac{1}{2} \frac{V^2}{\left|dV\right|^2}
\left(1-\frac{\mu}{n} \left|\frac{dV}{V}\right|^2\right)^2\right] 
g \kulk g,
\end{aligned} \]
where we used the fact that
\[
\frac{\tlhess }{V} 
= - \frac{\mu}{n} \left|\frac{dV}{V}\right|^2 g
\]
when restricted to $T \Sigma_0$.

We claim that the level set $V^{-1}(V_0)$ is connected. Assume that it
contains two connected components $\Sigma_0, \Sigma_1$. Since 
$V^{-1}(V_0, \infty)$ is connected and open, it is path connected so 
we can join $\Sigma_0$ and $\Sigma_1$ by a path $\gamma$ in 
$V^{-1}(V_0, \infty)$. If $v$ is larger than $V_0$ and the supremum of 
$V$ on $\gamma$, then $\Sigma_0$, $\Sigma_1$ and $\gamma$ are contained
in the same connected component $B$ of $V^{-1}(-\infty, v) \cup K$ which
is a ball. Then, for the gradient vector field $\nabla V$ the two 
hypersurfaces $\Sigma_0$ and $\Sigma_1$ are sources while the boundary
of $B$ is the only sink. Since $\nabla V$ has no zero outside
$V^{-1}(-\infty, v)$ this contradicts the Poincar\'e-Hopf theorem.

Note that our reasoning applies to any $v$ larger than $V_0$, so the
level sets $V^{-1}(v)$ are all round spheres.

From \eqref{eqdV} we conclude that $\left|\frac{dV}{V}\right|^2$ can
be expressed as a smooth function of $V$. We define a function
$s: \bH^n \setminus V^{-1}(V_0, \infty) \to \bR_+$ as 
$s \definedas f \circ V$ where 
\[
f(v) \definedas \int_{V_0}^v \frac{1}{|dV|}.
\]
Then $|ds| = 1$ so $s$ can be interpreted as the distance function
from $V^{-1}(V_0, \infty)$, see \cite{Petersen}. The second fundamental 
form of the level sets of $s$ is given by \eqref{eqSFF} so we see that 
the metric $g$ is rotationally symmetric.

Our next step is to prove that the conformal factor can be expressed
as a function of $s$.

We remark that we can reproduce the proof of Lemma \ref{lmUniformBound}
replacing $B_{R_0}$ by $V^{-1}(V_0-\epsilon, \infty)$ and find two functions
$f_\pm$ solving the Yamabe equation \eqref{eqYamabeb} such that
$f_- \leq U \leq f_+$ together with
$\left| \nabla^{(k)} (f_\pm - 1) \right| \leq A_k e^{-nr}$ for any integer
$k \geq 0$.

Since the conformal factor is bounded away from zero and from infinity,
the metrics $g$ and $b$ are uniformly equivalent. Hence, taking points
located further and further from $V^{-1}(V_0, \infty)$ with respect to
the hyperbolic metric yields points with $s$ going to infinity. This
proves that $s$ is unbounded.

The conformal transformation law of the mean curvature of the spheres
of constant $s$ is given by
\begin{equation}\label{eqMeanConf}
H^b U^{1-\kappa} = H^g U - \frac{2(n-1)}{n-2} \partial_s U.
\end{equation}

We choose a coordinate chart $(\theta^\mu)$ on the sphere and use it to
define Fermi coordinates on $\bH^n \setminus V^{-1}(V_0, \infty)$, so that
\begin{equation}\label{eqWarped}
g = ds^2 + f(s) \sigma_{\mu\nu} d\theta^\mu d\theta^\nu.
\end{equation}
From our previous discussion, both $H^g$ and $H^b$ are functions of $s$
only so it follows from \eqref{eqMeanConf}
that for any $\mu$,
\[
(1-\kappa) H^b U^{-\kappa} \partial_\mu U 
= H^g \partial_\mu U - \frac{2(n-1)}{n-2} \partial_s \partial_\mu U.
\]

As $s$ increases, the spheres of constant $s$ become larger and larger
and located further and further from $V^{-1}(V_0, \infty)$ so 
$H^b \to n-1$. From Formula \eqref{eqMeanConf} and the estimate on 
$U$ the same is true for $H^g$. As a consequence, the previous equation 
for $\partial_\mu U$ can be written as
\[
 \partial_s \partial_\mu U = (2 + o(1)) \partial_\mu U. 
\]

In particular, $\partial_\mu U$ grows as $e^{2s}$ unless $\partial_\mu U = 0$.
Such a growth is inconsistent with the decay assumption
$|\nabla U| = O(e^{-nr})$. This implies that $U$ is constant on the level
sets of $s$.

Without loss of generality, we can assume that the level set $V=V_0$ 
is a sphere of radius $R_1$ centered at the origin of the hyperbolic 
space. From Propositions \ref{propAdSSchPos2} and \ref{propAdSSchNeg}, 
there are constants $m_-$ such that $f_{m_-}(R_1) = 0$ and $m_+$ such that
$f_{m_+}(r) \to \infty$ when $r \to R_1$. By the intermediate value theorem,
there exists $m \in (m_-, m_+)$ such that $f_m(R_1)$ equals the value of
$U$ on $B_{R_1}$. By uniqueness of the solution of the Yamabe equation
\eqref{eqYamabeb} with Dirichlet boundary values, we conclude that 
$U = f_m$ on $\bH^n \setminus B_{R_1}$. By analytic continuation, this 
equality must hold everywhere on $\bH^n \setminus K$.
\end{proof}

\section{A density result}
\label{secDensity}


In this second appendix, we show that any asymptotically hyperbolic 
metric satisfying the decay assumptions of the positive mass theorem 
can be approximated by metrics which are conformal to the hyperbolic 
metric outside a compact set, while changing the mass functional by an 
arbitrarily small amount. This result is a refinement of 
\cite[Proposition 6.2]{ChruscielDelay}.

\begin{proposition}\label{propDensity}
Let $(M, g)$ be a $C^{2, \alpha}_\tau$-asymptotically hyperbolic manifold
for $\alpha\in (0,1)$ and $\tau > 0$ meaning that there exists a diffeomorphism
\[
\Phi: M \setminus K \to \bH^n \setminus B_{R_0}
\]
such that $e \definedas \Phi_* g - b$ belongs to $C^{2, \alpha}_\tau(M, S^2M)$,
that is to say $e \in C^{2, \alpha}_{loc}(M, S^2M)$ is such that
\[
\left\|e \right\|_{C^{2, \alpha}_\delta(\bH^n\setminus B_{R_0}, S^2M)} 
\definedas 
\sup_{x \in \bH^n\setminus B_{R_0+1}} e^{\delta r(x)} 
\left\|e \right\|_{C^{2, \alpha}(B_1(x), S^2M)}
< \infty.
\]
Assume further that $\scal^g \in L^\infty$ and $\scal^g \geq -n(n-1)$.
Then for any $\epsilon > 0$, there exist $R > R_0$ and $\lambda_R$ such that
\begin{itemize}
\item 
$\left|\lambda_R - g\right|_g < \epsilon$;
\item 
$\Phi_* \lambda_R$ is conformal to $b$ outside $B_R$, that is
$\Phi_* \lambda_R = U^\kappa b$ with $U \to 1$ at infinity;
\item 
$\scal^{\lambda_R} \geq -n(n-1)$ and $\scal^{\lambda_R} = -n(n-1)$
on $\bH^n \setminus B_R$.
\end{itemize}

In addition, assuming that $\tau > \frac{n}{2}$ and 
$\int_M (\scal^g +n(n-1)) \cosh r \, d\mu^g < \infty$,
we can also ensure that 
\[
\left| H_{\Phi}^{\lambda_R} (V_{(i)}) - H_{\Phi}^g (V_{(i)}) \right| < \epsilon
\]
for $i=0, \ldots, n$.
\end{proposition}

Note that if $(M,g)$ is an asymptotically hyperbolic manifold in the 
above sense and $E$ is a geometric tensor bundle over $M$ then one can 
define weighted H\"older spaces $C_\delta^{k,\alpha}(M,E) \definedas 
\{e^{-\delta r} u \mid u \in C^{k,\alpha}(M,E)\}$ with respective norms given 
by $\left\|u \right\|_{C^{k, \alpha}_\delta(M,E)} \definedas 
\left\| e^{\delta r(x)} u \right\|_{C^{k, \alpha}(M,E)}$. We refer the reader to 
\cite{LeeFredholm} for more details on these spaces.

\begin{proof}
We select a smooth cut-off function $\chi: \bR \to \bR$ such that
$\chi \equiv 1$ on $(-\infty, 0)$ and $\chi \equiv 0$ on $(1, \infty)$.
We let $r$ denote the distance from the origin in $\bH^n$ and
set $\chi_R(x) \definedas \chi(r(x) - R)$ for $R > 0$. We define
the metric $g_R$ by
\[
g_R \definedas \chi_R g + (1-\chi_R) \Phi^* b.
\] 
To prove the theorem we construct a function $v_R$ such that the metric 
$\lambda_R \definedas (1+v_R)^\kappa g_R$ and show that $\lambda_R$ is as 
close as we want to $g$ provided that $R$ is large enough.

To simplify notation, we set $\hscal^\lambda \definedas \scal^\lambda + n(n-1)$
for any metric $\lambda$ on $M$. We first remark that the
scalar curvatures of $g_R$ and $\lambda_R$ are related through
\[
-\frac{4(n-1)}{n-2} \Delta^{g_R} v_R + \scal^{g_R} (1+v_R) 
= \scal^{\lambda_R} (1+v_R)^{\kappa+1}.
\]
This equation can be rewritten as
\begin{equation} \label{eqPrescScal}
\begin{split}
&\frac{4(n-1)}{n-2} \left(-\Delta^{g_R} v_R + n v_R\right) 
+ n(n-1) \left[(1+v_R)^{\kappa+1} - 1 - (\kappa+1) v_R \right] \\
&\qquad 
+ \hscal^{g_R} v_R
= \hscal^{\lambda_R} (1+v_R)^{\kappa+1} - \hscal^{g_R}.
\end{split}
\end{equation}
To construct the function $v_R$ we introduce the following auxiliary
equation,
\begin{equation} \label{eqAuxiliary}
\frac{4(n-1)}{n-2} \left(-\Delta^{g_R} v_R + n v_R\right) + n(n-1) f(v_R)
+ \hscal^{g_R}~v_R = \chi_R \hscal^g - \hscal^{g_R},
\end{equation}
where we use the notation 
\[
f(x) = (1+x)^{\kappa+1} - 1 - (\kappa+1) x.
\]
Note that if $v_R > -1$ satisfies \eqref{eqAuxiliary} we have
\[
\hscal^{\lambda_R} (1+v_R)^{\kappa+1} = \chi_R \hscal^g
\]
from \eqref{eqPrescScal}. In particular, $\hscal^{\lambda_R} \geq 0$ and 
$\hscal^{\lambda_R} = 0$ on $\bH^n \setminus B_{R+1}$. That is to say, the 
metric $\lambda_R$ satisfies the second and the third assumptions of the 
theorem, provided that $v_R\to 0$ at infinity. We prove the existence of the function $v_R$ by the standard 
monotonicity method. We first remark that since $g$ and $g_R$ coincide 
inside $B_R$, the right hand side of \eqref{eqAuxiliary} has support in 
the annulus $A_{R, R+1}$.

From the fact that $e \definedas \Phi_* g - b$ belongs to 
$C^{2, \alpha}_\tau$, one can easily conclude that
\[
\left| \chi_R \hscal^g - \hscal^{g_R} \right| \leq C e^{-\tau R}
\]
for some constant $C$ depending only on $\left\| e \right\|_{C^{2, \alpha}_\tau}$.
In particular, given $\epsilon > 0$ small enough, the functions 
$v_R^{\pm} = \pm \epsilon$ are barriers for \eqref{eqAuxiliary}, that is,
\[
\left\lbrace
\begin{aligned}
\frac{4(n-1)}{n-2} \left(-\Delta^{g_R} v_R^+ + n v_R^+\right) 
+ n(n-1) f(v_R^+) + \hscal^{g_R}~v_R^+
& \geq \chi_R \hscal^g - \hscal^{g_R},\\
\frac{4(n-1)}{n-2} \left(-\Delta^{g_R} v_R^- + n v_R^-\right) 
+ n(n-1) f(v_R^-) + \hscal^{g_R}~v_R^-
& \leq \chi_R \hscal^g - \hscal^{g_R}.
\end{aligned}
\right.
\]
As a consequence, there exists a function $v_R$ satisfying 
\eqref{eqAuxiliary} and  
\[
-\epsilon \leq v_R \leq \epsilon,
\] 
see for example \cite[Proposition 2.1]{Gicquaud} for more details.

Since $\lambda_R = (1+v_R)^\kappa g$ inside $B_R$, we immediately get
that $\left|\lambda_R - g\right| \leq C \epsilon$ in this region. 
Outside $B_R$, we can just use the fact that $e \in C^{2, \alpha}_\tau$ and 
conclude
\[
\left|\lambda_R - g\right| \leq 
\left|\lambda_R - b\right| + \left|g - b\right| \leq 2 \epsilon
\]
if $R$ is large enough.

From standard analysis on asymptotically hyperbolic manifolds it 
follows that $v_R \in C^{2, \alpha}_{n}$. Since $e \in C^{2, \alpha}_\tau$ 
we see that 
$\left\|\chi_R \hscal^g - \hscal^{g_R}\right\|_{C^{0, \alpha}_{\tau'}} \to 0$
as $R \to \infty$ for any $\tau' \in \left(\frac{n}{2}, \tau\right)$.
This implies that $\left\| v_R \right\|_{C^{2, \alpha}_{\tau'}} \to 0$.

Let $e_R \definedas \Phi_* \lambda_R - b$. From the previous estimate
and the fact that $e = \Phi_*g - b \in C^{2, \alpha}_\tau$, we deduce that
$\left\| e_R - e \right\|_{C^{2, \alpha}_{\tau'}} \to 0$. We choose an 
arbitrary $R_1 > R_0$. As in the proof of Lemma \ref{lmMassLambda} we 
use formulas from \cite[page 114]{HerzlichMassFormulae} or 
\cite{ChruscielHerzlich} to write
\[\begin{split}
&H_{\Phi}^{\lambda_R}(V_{(i)}) - H_{\Phi}^g(V_{(i)}) \\
&\qquad = 
\int_{S_{R_1}} \Big( 
V_{(i)} \left[\divg^b (e_R-e) - d \tr^b (e_R-e) \right] \\
&\qquad \qquad \qquad \qquad \qquad \qquad
+ \tr^b (e_R-e) dV_{(i)} - (e_R-e)(\nabla^b V_{(i)}, \cdot) \Big) 
(\nu_{R_1}) \, d \mu^b \\
&\qquad \qquad 
+ \int_{\bH^n \setminus B_{R_1}} 
\left(
V_{(i)} \left(\hscal^{\lambda_R} - \hscal^g\right) 
+ \mathcal{Q}(e_R, V_{(i)}) - \mathcal{Q}(e, V_{(i)})
\right) 
\, d \mu^b \\
\end{split}\]                                                 
for $i=0, \ldots, n$. From this expression it follows that it suffices 
to prove that $\int_{\bH^n \setminus B_{R_1}} V_{(i)} 
\left(\hscal^{\lambda_R} - \hscal^g\right)\, d \mu^b \to 0$ when 
$R \to \infty$ to get that the mass vector of $\lambda_R$ converges
to that of $g$ as $R$ goes to infinity. This follows immediately from
\[
\left| \int_{\bH^n \setminus B_{R_1}} 
V_{(i)} \left( \hscal^{\lambda_R} - \hscal^g \right) \, d \mu^b \right|
\leq
\int_{\bH^n \setminus B_{R_1}} V_{(0)} \hscal^g 
\left| \frac{\chi_R}{(1+v_R)^{1+\kappa}} - 1 \right|\, d \mu^b
\]
and the fact that $\frac{\chi_R}{(1+v_R)^{1+\kappa}} - 1$ is uniformly
bounded for $R$ large enough and converges to $0$ almost everywhere.
\end{proof}

\bibliographystyle{amsplain}
\bibliography{../biblio-AH-mass}

\end{document}